\documentclass[11pt]{article}
\usepackage{smile}

\usepackage{fullpage}
\usepackage{lscape}
\usepackage{bigints}
\usepackage{framed}
\usepackage{mdframed}
\usepackage{enumerate}
\usepackage[inline]{enumitem}
\usepackage[T1]{fontenc}
\usepackage{moresize}
\usepackage{bm}
\usepackage{bbm}
\usepackage{dsfont}
\usepackage{amsmath}
\usepackage{amssymb}
\usepackage{amsthm}
\usepackage{amsfonts}
\usepackage{stmaryrd}
\usepackage{array}
\usepackage{mathrsfs}
\usepackage{mathtools} 
\usepackage{extarrows}
\usepackage{stackrel}
\usepackage{relsize,exscale}
\usepackage{scalerel}
\usepackage[nodisplayskipstretch]{setspace}
\usepackage{color}
\usepackage[usenames,dvipsnames]{xcolor}
\usepackage{cancel}
\usepackage{soul}
\usepackage{undertilde}
\usepackage{xfrac}
\usepackage{siunitx}
\usepackage{graphicx}
\usepackage{float}
\usepackage{rotating}
\usepackage{subcaption}
\usepackage{overpic}
\usepackage[all]{xy}
\usepackage{tikz}
\usetikzlibrary{arrows,matrix,positioning,calc,automata,patterns}
\usepackage{booktabs}
\usepackage{dcolumn}
\usepackage{multirow}
\usepackage{diagbox}
\usepackage{tabularx}
\usepackage{verbatim}
\usepackage{listings}
\usepackage[ruled,vlined]{algorithm2e}
\usepackage{fancyvrb}
\usepackage{hyperref}
\usepackage[round]{natbib}
\usepackage{sectsty}

\hypersetup{
    bookmarks=true,         
    unicode=false,          
    pdftoolbar=true,        
    pdfmenubar=true,        
    pdffitwindow=false,     
    pdfstartview={FitH},    
    pdftitle={My title},    
    pdfauthor={Author},     
    pdfsubject={Subject},   
    pdfcreator={Creator},   
    pdfproducer={Producer}, 
    pdfkeywords={key1, key2}, 
    pdfnewwindow=true,      
    colorlinks=true,        
    linkcolor=blue,         
    citecolor=blue,         
    filecolor=blue,         
    urlcolor=cyan           
}

\def\scaleint#1{\vcenter{\hbox{\scaleto[3ex]{\displaystyle\int}{#1}}}}

\usepackage{stackengine}
\stackMath
\newcommand\tenq[2][1]{%
\def\useanchorwidth{T}%
\ifnum#1>1%
\stackunder[0pt]{\tenq[\numexpr#1-1\relax]{#2}}{\!\scriptscriptstyle\thicksim}%
\else%
\stackunder[1pt]{#2}{\!\scriptstyle\thicksim}%
\fi%
}

\makeatletter
\DeclareRobustCommand\widecheck[1]{{\mathpalette\@widecheck{#1}}}
\def\@widecheck#1#2{%
    \setbox\z@\hbox{\m@th$#1#2$}%
    \setbox\tw@\hbox{\m@th$#1%
       \widehat{%
          \vrule\@width\z@\@height\ht\z@
          \vrule\@height\z@\@width\wd\z@}$}%
    \dp\tw@-\ht\z@
    \@tempdima\ht\z@ \advance\@tempdima2\ht\tw@ \divide\@tempdima\thr@@
    \setbox\tw@\hbox{%
       \raise\@tempdima\hbox{\scalebox{1}[-1]{\lower\@tempdima\box
\tw@}}}%
    {\ooalign{\box\tw@ \cr \box\z@}}}
\makeatother

\def\given{\,|\,}

\def\Biggiven{\,\Big{|}\,}
\def\tr{\mathop{\text{tr}}\kern.2ex}
\def\tZ{{\tilde Z}}

\def\tR{{\tilde r}}

\def\tW{{\tilde W}}

\def\tR{{\tilde R}}
\def\tL{{\tilde L}}
\def\P{{\mathrm P}}

\def\E{{\mathrm E}}

\def\d{{\mathrm d}}

\newcommand{\zahl}[1]{\llbracket #1\rrbracket}
\newcommand\yestag{\addtocounter{equation}{1}\tag{\theequation}}
\newcolumntype{L}[1]{>{\raggedright\let\newline\\\arraybackslash\hspace{0pt}}m{#1}}
\newcolumntype{C}[1]{>{  \centering\let\newline\\\arraybackslash\hspace{0pt}}m{#1}}
\newcolumntype{R}[1]{>{ \raggedleft\let\newline\\\arraybackslash\hspace{0pt}}m{#1}}
\newcolumntype{d}[1]{D{.}{.}{#1}}
\newcolumntype{H}{>{\setbox0=\hbox\bgroup}c<{\egroup}@{}}
\newcolumntype{Z}{>{\setbox0=\hbox\bgroup}c<{\egroup}@{\hspace*{-\tabcolsep}}}
\newcolumntype{b}{X}
\newcolumntype{s}{>{\hsize=.5\hsize}X}

\numberwithin{equation}{section}

\newtheorem{theorem}{Theorem}[section]
\newtheorem{lemma}{Lemma}[section]
\newtheorem{proposition}{Proposition}[section]

\newtheorem{innercustomgeneric}{\customgenericname}
\providecommand{\customgenericname}{}
\newcommand{\newcustomtheorem}[2]{%
  \newenvironment{#1}[1]
  {%
   \renewcommand\customgenericname{#2}%
   \renewcommand\theinnercustomgeneric{##1}%
   \innercustomgeneric
  }
  {\endinnercustomgeneric}
}
\newcustomtheorem{customdefinition}{Definition}
\newcustomtheorem{customdefinitions}{Definitions}
\newcustomtheorem{customtheorem}{Theorem}
\newcustomtheorem{customassumption}{Assumption}
\newcustomtheorem{customlemma}{Lemma}
\newcustomtheorem{customexample}{Example}
\theoremstyle{definition}

\usepackage{enumitem}
\makeatletter
\newcommand{\mylabel}[2]{#2\def\@currentlabel{#2}\label{#1}}
\makeatother

\graphicspath{{./fig3/}}

\allowdisplaybreaks

\begin{document}

\setlength{\abovedisplayskip}{5pt}
\setlength{\belowdisplayskip}{5pt}
\setlength{\abovedisplayshortskip}{5pt}
\setlength{\belowdisplayshortskip}{5pt}
\hypersetup{colorlinks,breaklinks,urlcolor=blue,linkcolor=blue}

\title{\LARGE On the failure of the bootstrap for Chatterjee's rank correlation}

\author{Zhexiao Lin\thanks{Department of Statistics, University of California, Berkeley, CA 94720, USA; e-mail: {\tt zhexiaolin@berkeley.edu}}~~~and~
Fang Han\thanks{Department of Statistics, University of Washington, Seattle, WA 98195, USA; e-mail: {\tt fanghan@uw.edu}}
}

\date{\today}

\maketitle

\vspace{-1em}

\begin{abstract}
While researchers commonly use the bootstrap for statistical inference, many of us have realized that the standard bootstrap, in general, does not work for Chatterjee's rank correlation. In this paper, we provide proof of this issue under an additional independence assumption, and complement our theory with simulation evidence for general settings. Chatterjee's rank correlation thus falls into a category of statistics that are asymptotically normal but bootstrap inconsistent. Valid inferential methods in this case are Chatterjee's original proposal (for testing independence) and \cite{lin2022limit}'s analytic asymptotic variance estimator (for more general purposes). 
\end{abstract}


\section{Introduction}\label{sec:intro}

Rank correlation is an essential tool for measuring the association between random variables. Its development is closely linked to the history of statistics as a discipline and has involved many notable figures, including Charles Spearman, Sir Maurice Kendall, Wassily Hoeffding, Jack Kiefer, Murray Rosenblatt, Jaroslav H\'{a}jek, Eric Lehmann, Herman Chernoff, and Richard Savage \citep{Spearman1904,Spearman1906,Kendall1938,kendall1948rank,MR0004426,MR0029139,Hoeffding1994,MR0125690,MR1680991,MR79383,MR100322}. Unlike other correlation coefficients, a rank correlation relies solely on the rankings of the original data, making it (1) exactly distribution-free when testing independence of continuous random variables; (2) invariant to marginal monotonic transformations (cf. copulas); and (3) robust in the face of outliers and heavy-tailedness. Its usefulness, therefore, is self-explanatory.

Given the remarkable progress made in this area over the past century, it is impressive that, recently, Sourav Chatterjee devised a new rank correlation that is both new and appealing from multiple perspectives. Specifically, consider an independent and identically distributed (i.i.d.) sample $\{X_i,Y_i\}_{i=1,\ldots,n}$ from a pair of scalars, $(X,Y)$, with joint and marginal distribution functions $F_{X,Y}$ and $F_X$, $F_Y$, respectively. Let 
\[
R_i:=\sum_{j=1}^n\ind(Y_j\leq Y_i)~~~ {\rm and}~~~ L_i:= \sum_{j=1}^n\ind(Y_j\ge Y_i) 
\]
be the rank and ``reversed'' rank of $Y_i$ with $\ind(\cdot)$ representing the indicator function, and let $\{[i],i=1,\ldots,n\}$ be a rearrangement of $\{1,\ldots,n\}$ such that $X_{[1]} \le \cdots \le X_{[n]}$ with ties broken at random. \cite{chatterjee2020new} introduced the following statistic
\begin{align}\label{eq:barxin*}
  \xi_n &:= 1 - \frac{n}{2\sum_{i=1}^n L_i (n-L_i)} \sum_{i=1}^{n-1} \Big\lvert R_{[i+1]} - R_{[i]} \Big\rvert,
\end{align}
which he showed to be a strongly consistent estimator of Dette-Siburg-Stoimenov's dependence measure \citep{MR3024030},
\begin{align}\label{eq:xi}
\xi=\xi(X,Y):=\;&  \frac{\scaleint{4.5ex}\,\Var\big\{\E\big[\ind\big(Y\geq y\big) \given X \big] \big\} \d F_{Y}(y)}{\scaleint{4.5ex}\,\Var\big\{\ind\big(Y\geq y\big)\big\}\d F_{Y}(y)},
\end{align}
as long as $Y$ is not almost surely a constant. 

Why is $\xi_n$ in \eqref{eq:barxin*} appealing? There are three reasons Sourav Chatterjee outlines. First, it has a simple form. Second, it has a normal limiting null distribution. Finally, it measures a dependence measure, $\xi$ in \eqref{eq:xi}, that satisfies R\'{e}nyi's criteria \citep{MR0115203} and Bickel's definition of a measure of functional dependence \citep{bickel2022measures}: $\xi$ is zero if and only if $Y$ is independent of $X$ and one if and only if $Y$ is a measurable function of $X$. Therefore, $\xi_n$ is a rank correlation that can accurately quantify both independence and functional dependence. This is something that all aforementioned rank correlations, including Spearman's $\rho$, Kendall's $\tau$, Hoeffding's $D$, Blum-Kiefer-Rosenblatt's $r$, and Bergsma-Dassios-Yanagimoto's $\tau^*$ \citep{MR3178526,yanagimoto1970measures}, fail to achieve.

Due to these appealing properties, Chatterjee's rank correlation has gained significant interest and a wave of research has emerged exploring its applications and extensions. Notable recent works include \cite{azadkia2019simple}, \cite{cao2020correlations}, \cite{shi2020power}, \cite{gamboa2020global}, \cite{deb2020kernel}, \cite{huang2020kernel}, \cite{auddy2021exact}, \cite{shi2021ac}, \cite{lin2021boosting}, \cite{fuchs2021bivariate}, \cite{azadkia2021fast}, \cite{griessenberger2022multivariate}, \cite{strothmann2022rearranged}, \cite{zhang2022asymptotic}, \cite{bickel2022measures}, \cite{holma2022correlation}, \cite{chatterjee2022estimating}, \cite{lin2022limit}, \cite{han2022azadkia}, \cite{ansari2022simple}, and \cite{zhang2023relationships}. Additionally, brief surveys on recent progress of (Chatterjee's and more) rank correlation methods have been conducted in \cite{han2021extensions} and \cite{chatterjee2022survey}.

This paper aims to investigate the validity of the standard bootstrap \citep{efron1979bootstrap,efron1981nonparametric} when applied to {\it fixed and continuous} $F_{X,Y}$, with $\xi_n$ taking the form \eqref{eq:barxin*} to handle ties in resampled data. We prove, in the simple independence case with $F_{X,Y}=F_XF_Y$, that the standard bootstrap results in an inconsistent estimator of $\xi_n$'s asymptotic variance, and the bootstrap distribution fails to converge to the limiting distribution of $\sqrt{n}(\xi_n-\xi)$. Simulations further complement the theory, indicating that the standard bootstrap will likely also fail in general settings with $F_{X,Y}\ne F_XF_Y$, even though $\sqrt{n}(\xi_n-\xi)$ still weakly converges to a normal distribution  \citep{lin2022limit}. Chatterjee's rank correlation thus falls into a class of statistics that are root-$n$ consistent, asymptotically normal, but bootstrap inconsistent --- a class including Bickel and Freedman's U-statistics \citep[Section 6]{bickel1981some}, Hodges estimator \citep[Pages 213-214]{beran1982estimated}, and  Abadie and Imbens's matching estimator \citep{abadie2008failure}.

It is worth mentioning that there are valid alternatives to using the standard bootstrap for inferring $\xi$ from $\xi_n$. Sourav Chatterjee derived the limiting null distribution of $\xi_n$ for testing independence between $X$ and $Y$ in \cite{chatterjee2020new}. In \cite{lin2022limit}, conditions were proven under which $\xi_n$ is root-$n$ consistent and asymptotically normal, and an analytic estimator of its asymptotic variance was proposed. This paper thus helps to justify the value of these derivations by demonstrating the inconsistency of an otherwise attractive alternative.

\section{Main results}

\subsection{Setup}

In this paper we consider the standard model of $(X_1,Y_1), \ldots, (X_n,Y_n)$ to be $n$ independent copies of $(X,Y)$ drawn from a fixed and continuous $F_{X,Y}$ of support in the two-dimensional real space. In this case, with probability one there is no tie in the observation, and hence $\xi_n$ admits the following simpler form,
\begin{align}\label{eq:barxin}
\xi_n &:= 1 - \frac{3}{n^2-1} \sum_{i=1}^{n-1} \Big\lvert R_{[i+1]} - R_{[i]} \Big\rvert.
\end{align}

To implement the standard bootstrap, consider $(\mX_b,\mY_b) = \{(X_{b,i},Y_{b,i})\}_{i=1}^n$ to be the bootstrap sample, of size $n$ and likely embracing ties, by sampling with replacement from $(\mX,\mY):=\{(X_i,Y_i)\}_{i=1}^n$. Let $\tilde{\xi}_b$ be Chatterjee's rank correlation calculated using the bootstrap sample $(\mX_b,\mY_b)$. More specifically, write
\begin{align}\label{eq:ib}
\Big\{[i]_b,i=1,\ldots,n\Big\} 
\end{align}
to be a rearrangement of $\{1,\ldots,n\}$ such that $X_{b,[1]_b} \le \cdots \le X_{b,[n]_b}$, with ties in $\{X_{b,i}\}_{i=1}^n$ broken {\it in arbitrary way}. Let 
\[
R_{b,i}:=\sum_{j=1}^n\ind(Y_{b,j}\leq Y_{b,i})~~~ {\rm and} ~~~L_{b,i}:= \sum_{j=1}^n\ind(Y_{b,j}\ge Y_{b,i})
\]
be the ranks and reversed ranks of the bootstrap sample. The bootstrapped rank correlation $\tilde{\xi}_b$ is then defined to be 
\begin{align}\label{eq:xib}
  \tilde{\xi}_b &:= 1 - \frac{n}{2\sum_{i=1}^n L_{b,i} (n-L_{b,i})} \sum_{i=1}^{n-1} \Big\lvert R_{b,[i+1]_b} - R_{b,[i]_b} \Big\rvert,
\end{align}
namely, plugging \eqref{eq:barxin*} to the bootstrap sample.

In this paper, we investigate two commonly used versions of the bootstrap in empirical research. The first version involves centering the bootstrap sample at $\xi_n$, which is calculated using the original sample. The second version involves centering the bootstrap sample at the mean of the bootstrap distribution $\E[\tilde{\xi}_b \given \mX,\mY]$. One can then estimate the asymptotic variance using 
\[
{\rm either}~~~n\E[(\tilde{\xi}_b - \xi_n)^2 \given \mX,\mY]~~~{\rm or}~~~n\Var[\tilde{\xi}_b \given \mX,\mY],
\]
assuming an infinite number of replications for bootstrap. In addition, it is of interest to evaluate the closeness of the bootstrap distributions of 
\[
\sqrt{n}(\tilde{\xi}_b - \xi_n)\given \mX,\mY~~~{\rm and}~~~~ \sqrt{n}(\tilde{\xi}_b - \E[\tilde{\xi}_b\given\mX,\mY]) \given \mX,\mY
\]
to that of $\sqrt{n}(\xi_n-\xi)$.

\subsection{Theory}

Our theory section has to be focused on the simple independence case with $F_{X,Y}=F_XF_Y$, only under which we are able to provide the otherwise formidable calculation of the limits of $n\E[(\tilde{\xi}_b - \xi_n)^2]$ and $n\Var[\tilde{\xi}_b \given \mX,\mY]$. 

The following result of Chatterjee establishes the limiting distribution of $\xi_n$ under independence.

\begin{proposition}[Theorem 2.1, \cite{chatterjee2020new}]\label{thm:chatterjee}
Assume $F_{X,Y}=F_XF_Y$ is fixed and continuous. We then have $\sqrt{n}\xi_n$ weakly converges to $N(0,2/5)$.
\end{proposition}

Below is the main theorem of this paper. 

\begin{theorem}[Bootstrap inconsistency]\label{thm:main}  
    Assuming the same conditions of Proposition \ref{thm:chatterjee}, the following two statements then hold.
 \begin{enumerate}[label=(\roman*)]    
    \item (variance inconsistency)  $n\E[(\tilde{\xi}_b - \xi_n)^2 \given \mX,\mY]$ and $n\Var[\tilde{\xi}_b \given \mX,\mY]$ do not converge to $2/5$ in probability.
    \item (distribution inconsistency) There exists a sequence of measurable events $[\cE_i]_{i=1}^\infty$, satisfying 
    \[
    \liminf_{n \to \infty} \P((X_1,Y_1, \ldots,X_n,Y_n) \in \cE_n)>0, 
    \]
    such that the distributions of $\sqrt{n}(\tilde{\xi}_b - \xi_n)$ and $\sqrt{n}(\tilde{\xi}_b - \E[\tilde{\xi}_b\given\mX,\mY])$ do not converge to $N(0,2/5)$ conditional on $(X_1,Y_1,\ldots,X_n,Y_n) \in \cE_n$.
  \end{enumerate}
\end{theorem}

The authors are intrigued by Theorem \ref{thm:main} and believe its significance is best appreciated in the context of mathematical statistics history, where the bootstrap method's validity has been a central topic, with establishing/disproving its consistency being particularly imperative. 

For an i.i.d. sample, bootstrap consistency is often linked to the studied statistic's root-$n$ consistency and asymptotic normality. According to \citet[Page 128]{shao2012jackknife}, the conventional wisdom seems to suggest that ``[u]sually the consistency of the bootstrap distribution estimator requires some smoothness conditions that are {\it almost the same as} those required for the asymptotic normality of the given statistic and certain moment conditions.'' As a matter of fact, this insight has been partly formalized in \citet[Theorem 1]{mammen2012does}, where Enno Mammen demonstrated, elegantly, that bootstrap consistency is equivalent to asymptotic normality when applied to linear functionals.

Indeed, the majority of theoretical results on bootstrap inconsistency are centered on statistics that do not exhibit a regular pattern of being root-$n$ consistent and asymptotically normal. In this regard, \cite{athreya1987bootstrap}, \cite{knight1989bootstrap}, and \cite{hall1990asymptotic} focused on the sample mean with a sample drawn from heavy-tailed distributions, \cite{beran1985bootstrap} on eigenvalues, \cite{hall1993inconsistency} on ranked parameters, and \cite{andrews2000inconsistency} and \cite{drton2011quantifying} on parameters at the boundary. Furthermore, \cite{abrevaya2005bootstrap}, \cite{kosorok2008bootstrapping}, and \cite{sen2010inconsistency} examined cubic-root consistent estimators, \cite{bretagnolle1983lois} and \cite{arcones1992bootstrap} investigated degenerate U- and V-statistics, and \cite{dumbgen1993nondifferentiable} and \cite{fang2019inference} explored a general class of non-smooth plug-in estimators.

For statistics that exhibit root-$n$ consistency and asymptotic normality, we categorize the cases where bootstrap inconsistency arises to three groups: (1) those that fail due to moment condition, e.g., Bickel and Freedman's U-statistics \citep[Section 6]{bickel1981some}; (2) those that fail at superefficiency points, e.g., Hodges and Stein estimators \citep{beran1997diagnosing,samworth2003note}; and (3) those that do not belong to the previous two groups, including, notably, Abadie and Imbens's nearest neighbor matching estimator of the average treatment effect \citep{abadie2008failure}. 

Abadie and Imbens's case is particularly relevant to our work on Chatterjee's rank correlation, as both can be perceived as a type of nearest neighbor graph-based statistics with a {\it fixed} number of nearest neighbors. To the best of the authors' knowledge, however, no work has established a general relationship between the inconsistency of the bootstrap and the ``irregularity'' of graph-based statistics; this would be an interesting future question for mathematical statisticians.

Finally, it should be noted that the issue of bootstrap inconsistency is not a universal problem affecting all rank correlations or rank-based statistics. For example, the bootstrap consistency of Spearman's $\rho$ and Kendall's $\tau$ can be easily established based on the works of \cite{bickel1981some} and \cite{arcones1992bootstrap}. On the other hand, the bootstrap inconsistency of Hoeffding's $D$, Blum-Kiefer-Rosenblatt's $r$, and Bergsma-Dassios-Yanagimoto's $\tau^*$ under independence between $X$ and $Y$ is caused by the non-normal convergence of the degenerate U-statistics, but not by the ranking.

\subsection{Simulations}

One might be tempted to speculate that the bootstrap inconsistency observed in Theorem \ref{thm:main} is solely due to the ``degeneracy'' property of the null point of independence. While independence does play a crucial role in the cases of Hoeffding's $D$, Blum-Kiefer-Rosenblatt's $r$, and Bergsma-Dassios-Yanagimoto's $\tau^*$ --- they only become degenerate when $X$ is independent of $Y$ --- the case of Chatterjee's rank correlation appears to be different. 

As tracking the limits of $n\E[(\tilde{\xi}_b - \xi_n)^2]$ and $n\Var[\tilde{\xi}_b \given \mX,\mY]$ under dependence between $Y$ and $X$ is technically intimidating, this paper relies on simulations to illustrate this point. To this end, we investigate the following six methods.
\begin{enumerate}[label=(\roman*)]
\item (V-LH) The asymptotic variance estimator described in \citet[Theorem 1.2]{lin2022limit};
\item (V-B1) the bootstrap asymptotic variance estimator using $n\E[(\tilde{\xi}_b - \xi_n)^2 \given \mX,\mY]$;
\item (V-B2) the bootstrap asymptotic variance estimator using $n\Var[\tilde{\xi}_b \given \mX,\mY]$;
\item (D-LH) constructing the confidence interval using the idea described in \citet[Remark 1.4]{lin2022limit};
\item (D-HB1) constructing the confidence interval using the hybrid bootstrap \citep[Section 4.1.5]{shao2012jackknife} based on $\sqrt{n}(\tilde{\xi}_b - \xi_n)\given \mX,\mY$;
\item (D-HB2) constructing the confidence interval using the hybrid bootstrap based on $\sqrt{n}(\tilde{\xi}_b - \E[\tilde{\xi}_b\given\mX,\mY]) \given \mX,\mY$.
\end{enumerate}

The simulation studies were conducted based on the Gaussian rotation model, where $(X,Y)$ are bivariate Gaussian with mean 0 and covariance matrix $\Sigma$, defined as
\[
 \Sigma = \bigg(
  \begin{matrix}
    1 &~~ \rho\\
    \rho &~~ 1
  \end{matrix} \bigg),~~~{\rm with}~\rho\in (
  -1,1).
\]

We investigate the performance of different methods for estimating $\xi_n$'s variance and inferring $\xi$ using various sample sizes $n=1,000, 5,000, 10,000$ and population correlations $\rho=0, 0.3, 0.5, 0.7, 0.9$. For the bootstrap procedure, we adopt a bootstrap size of $5,000$ and simulate $5,000$ replications to compute the square roots of the mean squared errors (RMSEs) in estimating $n\Var(\xi_n)$ (of limits 0.4,  0.46, 0.51, 0.47, and 0.24 as $\rho$ changes from 0 to 0.9), as well as the empirical coverage probabilities with the nominal level $\alpha=0.05$ or 0.1.

Table \ref{tab:sim} presents the simulation results, demonstrating that regardless of the strength of dependence characterized by $\rho$, the bootstrap methods consistently produce erroneous variance estimators and inaccurate confidence intervals. On the other hand, the method proposed by \cite{lin2022limit} performs well for large $n$.

{
\renewcommand{\tabcolsep}{1.5pt}
\renewcommand{\arraystretch}{1.1}
\begin{table}[t!p]
\aboverulesep=0ex
\belowrulesep=0ex
\centering 
\caption{Variance estimation and empirical coverage probability}{
\begin{tabular}{C{.5in}C{.5in}C{.5in}C{.5in}C{.5in}C{.5in}C{.5in}C{.5in}C{.5in}C{.5in}C{.5in}}
\toprule
\multirow{2}{*}{$\rho$} & \multirow{2}{*}{$n$}
& \multicolumn{3}{c}{Variance, RMSE} & \multicolumn{3}{c}{Coverage, $\alpha=0.05$}&  \multicolumn{3}{c}{Coverage, $\alpha=0.1$} \\
\cline{3-11}
& &V-LH & V-B1 & V-B2 & D-LH & D-HB1 & D-HB2 & D-LH & D-HB1 & D-HB2\\
\midrule
0 & 1000 & 0.18 & 135.65 & 0.09 & 0.90 &   0.00 & 0.92 & 0.85 &   0.00 & 0.86 \\ 
& 5000 & 0.08 & 676.92 & 0.09 & 0.94 &   0.00 & 0.91 & 0.89 &   0.00 & 0.85 \\ 
& 10000 & 0.06 & 1353.32 & 0.09 & 0.95 &   0.00 & 0.92 & 0.90 &   0.00 & 0.85 \\ 
\midrule
0.3 & 1000 & 0.18 & 122.19 & 0.16 & 0.90 &   0.00 & 0.89 & 0.85 &   0.00 & 0.82 \\ 
& 5000 & 0.07 & 610.14 & 0.16 & 0.94 &   0.00 & 0.88 & 0.89 &   0.00 & 0.81 \\ 
& 10000 & 0.05 & 1220.34 & 0.16 & 0.95 &   0.00 & 0.89 & 0.90 &   0.00 & 0.82 \\ 
\midrule
0.5 & 1000 & 0.17 & 98.94 & 0.23 & 0.90 &   0.00 & 0.84 & 0.84 &   0.00 & 0.76 \\ 
& 5000 & 0.07 & 495.33 & 0.24 & 0.95 &   0.00 & 0.85 & 0.89 &   0.00 & 0.77 \\ 
& 10000 & 0.05 & 990.39 & 0.24 & 0.95 &   0.00 & 0.85 & 0.90 &   0.00 & 0.77 \\ 
\midrule
0.7 & 1000 & 0.15 & 65.81 & 0.26 & 0.91 &   0.00 & 0.81 & 0.84 &   0.00 & 0.72 \\ 
& 5000 & 0.06 & 329.11 & 0.27 & 0.95 &   0.00 & 0.81 & 0.89 &   0.00 & 0.73 \\ 
& 10000 & 0.04 & 657.98 & 0.26 & 0.95 &   0.00 & 0.82 & 0.91 &   0.00 & 0.74 \\ 
\midrule
0.9 & 1000 & 0.12 & 23.81 & 0.15 & 0.82 &   0.00 & 0.78 & 0.76 &   0.00 & 0.69 \\ 
& 5000 & 0.04 & 119.01 & 0.14 & 0.93 &   0.00 & 0.78 & 0.88 &   0.00 & 0.69 \\ 
& 10000 & 0.03 & 238.42 & 0.15 & 0.94 &   0.00 & 0.77 & 0.89 &   0.00 & 0.68 \\ 
\bottomrule
\end{tabular}}
\label{tab:sim}
\end{table}
}

\section{Proof of Theorem \ref{thm:main}}

We first introduce some necessary notation. For any integers $n\ge 1$, let $\zahl{n}:= \{1,2,\ldots,n\}$. A set consisting of distinct elements $x_1,\dots,x_n$ is written as either $\{x_1,\dots,x_n\}$ or $\{x_i\}_{i=1}^{n}$, and its cardinality is written by $\lvert \{x_i\}_{i=1}^n \rvert$. The corresponding sequence is denoted by $[x_1,\dots,x_n]$ or $[x_i]_{i=1}^{n}$. For any $a,b \in \bR$, write $a \vee b := \max\{a,b\}$ and $a \wedge b := \min\{a,b\}$. For any two real sequences $\{a_n\}$ and $\{b_n\}$, write $a_n = O(b_n)$ if $\limsup |a_n/b_n|$ is bounded.

Given a sample $\{(X_i,Y_i)\}_{i=1}^n$, let $W_i=W_{b,i}$ be the number of times $(X_i,Y_i)$ appears in the bootstrap sample $(\mX_b,\mY_b)$; $(W_1,\ldots,W_n)$ then follows a multinomial distribution with parameter $(n,1/n,\ldots,1/n)$, which we denote by $M_n(n;1/n,\ldots,1/n)$. Introduce 
\[
\tR_i := \ind(W_i>0)\cdot\sum_{j=1}^n W_j \ind(Y_j \le Y_i) 
\]
to represent the rank of $Y_i$ in the bootstrap sample if $(X_i,Y_i)$ appears in the bootstrap sample; note that $\tR_i$ is different from $R_{b,i}$ as they use different indexing system. 


Recall that $\{[i],i=1,\ldots,n\}$ is defined to be the unique  rearrangement of $\{1,\ldots,n\}$ such that 
\[
X_{[1]} < \cdots < X_{[n]} 
\]
as $F_X$ has been assumed to be continuous. One could then define
\begin{align}\label{eq:hatxib}
  \hat{\xi}_b &:= 1 - \frac{3}{n^2-1} \sum_{i=1}^{n-1} \ind\Big(W_{[i]}>0\Big)\Big\lvert \tR_{[i+k(i)]} - \tR_{[i]} \Big\rvert,
\end{align}
where 
\[
k(i) := 
\begin{cases}
        \min\{k:W_{[i+k]}>0\}, & \mbox{ if the left set is nonempty},\\   
        0, & \mbox{otherwise}.
    \end{cases} 
\]

Comparing \eqref{eq:hatxib} with \eqref{eq:xib}, we have the following key identity,
\begin{align}\label{eq:key}
  \sum_{i=1}^{n-1} \ind\Big(W_{[i]}>0\Big)\Big\lvert \tR_{[i+k(i)]} - \tR_{[i]} \Big\rvert = \sum_{i=1}^{n-1} \Big\lvert R_{b,[i+1]_b} - R_{b,[i]_b} \Big\rvert,
\end{align}
where $\{[i]_b,i=1,\ldots,n\}$, introduced in \eqref{eq:ib}, is defined to be a rearrangement of $\{1,\ldots,n\}$ such that $X_{b,[1]_b} \le \cdots \le X_{b,[n]_b}$, with ties in $\{X_{b,i}\}_{i=1}^n$ broken {\it in arbitrary way}. 

Why does \eqref{eq:key} hold? The reason is that, no matter how we adjust for the ties on $\{X_{b,i}\}_{i=1}^n$, it is always true that 
\[
Y_{b,i} = Y_{b,j} \text{ whenever } X_{b,i} = X_{b,j}
\]
since they are generated by the same unit from the bootstrap procedure. Then 
\[
R_{b,[i+1]_b} - R_{b,[i]_b} \text{ is nonzero only if } X_{b,[i+1]_b} \ne X_{b,[i]_b}, 
\]
or equivalently, 
\[
X_{b,[i+1]_b} \text{ is the right nearest neighbor (NN) of }X_{b,[i]_b}
\]
in the bootstrap sample. Note that the bootstrap procedure will not change the order of $X_1,\ldots,X_n$. In order to find the right NN of a unit in the bootstrap sample, it is then equivalent to finding the right NN in the original sample among those that also appear in the bootstrap sample.

The difference between \eqref{eq:hatxib} and \eqref{eq:xib} is, then, fully due to the replacement of a random quantity $n/[2\sum_{i=1}^n L_{b,[i]_b} (n-L_{b,[i]_b})]$ by a fixed quantity $3/(n^2-1)$, corresponding to the simple form of $\xi_n$ in \eqref{eq:barxin}. The following lemma shows that this difference is $\sqrt{n}$-negaliable. All lemmas stated in the sequel are under the conditions of Theorem~\ref{thm:main}.

\begin{lemma}\label{lemma:l2} It holds true that
  \begin{align}
    \lim_{n \to \infty} n\E[\hat{\xi}_b - \tilde{\xi}_b]^2 = 0.
  \end{align}
\end{lemma}

With lemma \ref{lemma:l2}, it then suffices to consider the asymptotic behavior of $\hat{\xi}_b$. However, \eqref{eq:hatxib} cannot be further simplified; for instance, one cannot replace $\tR_i$'s by $R_i$'s. This places more difficulty to analyzing \eqref{eq:hatxib}. Our route to proving Theorem \ref{thm:main} is instead built on a seemingly more complicated reformulation of $\hat\xi_b$:
\begin{align*}
  \overline{\xi}_b = 1 - \frac{3}{n^2-1} \sum_{i=1}^{n-1} \ind\Big(W_{[i]}>0\Big)\sum_{k=1}^{n-i}\Big\{\Big[ \sum_{j=1}^{n} W_j \ind\big(Y_{[i]} \wedge Y_{[i+k]} < Y_j < Y_{[i]} \vee Y_{[i+k]} \big) \Big] \\
  \ind\Big(W_{[i+\ell]}=0,\forall \ell \in \zahl{k-1}\Big) \ind\Big(W_{[i+k]}>0\Big)\Big\}.
\end{align*}.

The following lemma establishes the asymptotic equivalence between $\hat\xi_b$ and $\overline{\xi}_b$.


\begin{lemma}\label{lemma:han-l1} We have 
\[
\lvert \overline{\xi}_b - \hat{\xi}_b \rvert\leq \frac{6n}{n^2-1}=O(n^{-1}).
\]
\end{lemma}




With Lemma \ref{lemma:han-l1}, it then suffices to consider $\overline{\xi}_b$. Conditional on $\mX,\mY$, without loss of generality, we can assume $[X_i]_{i=1}^n$ is strictly increasing. One could then simply write
\begin{align*}
  \overline{\xi}_b = 1 - \frac{3}{n^2-1} \sum_{i=1}^{n-1} \ind\Big(W_i>0\Big)\sum_{k=1}^{n-i}\Big\{\Big[ \sum_{j=1}^{n} W_j \ind\big(Y_i \wedge Y_{i+k} < Y_j < Y_i \vee Y_{i+k} \big) \Big] \\
  \ind\Big(W_{i+\ell}=0,\forall \ell \in \zahl{k-1}\Big) \ind\Big(W_{i+k}>0\Big)\Big\}.
\end{align*}
For any $i \in \zahl{n-1}$ and $k \in \zahl{n-i}$, let
\begin{align*}
  S_{i,k} := \Big\{j \in \zahl{n}: Y_i \wedge Y_{i+k} < Y_j < Y_i \vee Y_{i+k} \Big\} \setminus \Big\{j\in\zahl{n}:i<j<i+k\Big\}.
\end{align*}
Then by the event $\{W_{i+\ell}=0,\forall \ell \in \zahl{k-1}\}$, we have
\begin{align*}
  \overline{\xi}_b =& 1 - \frac{3}{n^2-1} \sum_{i=1}^{n-1} \ind\Big(W_i>0\Big)\sum_{k=1}^{n-i}\Big[ \sum_{j\in S_{i,k}} W_j  \Big]  \ind\Big(W_{i+\ell}=0,\forall \ell \in \zahl{k-1}\Big) \ind\Big(W_{i+k}>0\Big).
\end{align*}

To continue the proof, we introduce the following lemma on some properties of the multinomial distribution.

\begin{lemma}\label{lemma:multi}
  Recall that $(W_1,\ldots,W_n) \sim M_n(n;1/n,\ldots,1/n)$. Let $\cS,\cS',\cT \subset \zahl{n}$ be mutually disjoint. Then
  \begin{align*}
    & \E \Big[\Big(\sum_{s\in\cS} W_s\Big) \ind\Big(\sum_{t \in \cT} W_t = 0\Big)\Big] = \lvert \cS \rvert \Big(1-\frac{\lvert \cT \rvert}{n}\Big)^{n-1};\\
    & \E \Big[\Big(\sum_{s\in\cS} W_s\Big)^2 \ind\Big(\sum_{t \in \cT} W_t = 0\Big)\Big] = \lvert \cS \rvert \Big(1-\frac{\lvert \cT \rvert}{n}\Big)^{n-1} + \lvert \cS \rvert^2 \Big(1-\frac{1}{n}\Big) \Big(1-\frac{\lvert \cT \rvert}{n}\Big)^{n-2};\\
  {\rm and}~~~  & \E \Big[\Big(\sum_{s\in\cS} W_s\Big) \Big(\sum_{s\in\cS'} W_s\Big) \ind\Big(\sum_{t \in \cT} W_t = 0\Big)\Big] = \lvert \cS \rvert \lvert \cS' \rvert \Big(1-\frac{1}{n}\Big) \Big(1-\frac{\lvert \cT \rvert}{n}\Big)^{n-2}.
  \end{align*}
\end{lemma}

Note that $S_{i,k}$ is disjoint with $\{i,i+1,\ldots,i+k\}$ almost surely since $[Y_i]_{i=1}^n$ are distinct with probability one. Then we can use Lemma~\ref{lemma:multi} to establish the following lemma on the expectation of $\overline{\xi}_b$ conditional on the original sample.

\begin{lemma}\label{lemma:exp,cond} We have
  \begin{align*}
    \E[\overline{\xi}_b \given \mX,\mY] = 1 - \frac{3}{n^2-1} \sum_{i=1}^{n-1} \sum_{k=1}^{n-i} \lvert S_{i,k} \rvert \Big[ \Big(1-\frac{k-1}{n}\Big)^{n-1} - 2 \Big(1-\frac{k}{n}\Big)^{n-1} + \Big(1-\frac{k+1}{n}\Big)^{n-1} \Big].
  \end{align*}
\end{lemma}

For any $j \in \zahl{n} \setminus \{i,i+1,\ldots,i+k\}$, by independence of $X$ and $Y$, 
\[
\P(Y_i \wedge Y_{i+k} < Y_j < Y_i \vee Y_{i+k}) = 1/3, 
\]
and thus 
\[
\E[\lvert S_{i,k} \rvert] = (n-k-1)/3. 
\]
Some calculations then yield the following lemma on the expectation of $\overline{\xi}_b$.

\begin{lemma}\label{lemma:exp} We have
  \begin{align*} 
    \E[\overline{\xi}_b] = \frac{1}{e} + O\Big(\frac{1}{n}\Big).
  \end{align*}
\end{lemma}

The variance of $\overline{\xi}_b$ conditional on the original sample can be established in a similar way by using Lemma~\ref{lemma:multi}. For any $i,j\in\zahl{n-1}$, $k\in\zahl{n-i}$ and $\ell \in \zahl{n-j}$, we write
\begin{align*}
  S_{i,k \cap j,\ell } =& \Big\{t\in\zahl{n}: (Y_i \wedge Y_{i+k}) \vee (Y_j \wedge Y_{j+\ell}) < Y_t < (Y_i \vee Y_{i+k}) \wedge (Y_j \vee Y_{j+\ell}) \Big\} \\
  & \setminus \Big\{t\in\zahl{n}:i<t<i+k,j<t<j+\ell\Big\};\\
  S_{i,k \setminus j,\ell } =& \Big\{t\in\zahl{n}: Y_i \wedge Y_{i+k} < Y_t < Y_i \vee Y_{i+k}, Y_t > Y_j \vee Y_{j+\ell} ~{\rm or}~ Y_t < Y_j \wedge Y_{j+\ell}\Big\} \\
  & \setminus \Big\{t\in\zahl{n}:i<t<i+k,j<t<j+\ell\Big\}.
\end{align*}

For any four integers $n,m,k,\ell$, define
\begin{align*}
   a_{n,m,k,\ell} :=& \Big(1-\frac{k+\ell-2}{n}\Big)^{m} - 3\Big(1-\frac{k+\ell-1}{n}\Big)^{m} + 3\Big(1-\frac{k+\ell}{n}\Big)^{m} - \Big(1-\frac{k+\ell+1}{n}\Big)^{m};\\
   b_{n,m,k,\ell} := &\Big(1-\frac{k+\ell-2}{n}\Big)^{m} - 4\Big(1-\frac{k+\ell-1}{n}\Big)^{m} + 6\Big(1-\frac{k+\ell}{n}\Big)^{m} - 4\Big(1-\frac{k+\ell+1}{n}\Big)^{m}\\
  & + \Big(1-\frac{k+\ell+2}{n}\Big)^{m};\\
   c_{n,m,k} :=& \Big(1-\frac{k-1}{n}\Big)^{m} - 2\Big(1-\frac{k}{n}\Big)^{m} + \Big(1-\frac{k+1}{n}\Big)^{m}.
\end{align*}
We can then use Lemma~\ref{lemma:multi} to establish the following result on the variance of $\overline{\xi}_b$ conditional on the original sample.
\begin{lemma}\label{lemma:var,cond}
  \begin{align*}
    & \Var[\overline{\xi}_b \given \mX,\mY]\\
    =& \frac{9}{(n^2-1)^2} \Big[\sum_{i=1}^{n-1} \sum_{k=1}^{n-i} \Big(\lvert S_{i,k} \rvert c_{n,n-1,k} + \lvert S_{i,k} \rvert^2 c_{n,n-2,k}(1-1/n)\Big) \\
    & + 2 \sum_{i,j=1,i<j}^{n-1} \sum_{k=1}^{n-i} \sum_{\ell=1}^{n-j} \Big(\lvert S_{i,k \cap j,\ell} \rvert [a_{n,n-1,k,\ell}\ind(k=j-i) + b_{n,n-1,k,\ell}\ind(k<j-i)] \\
    & + (\lvert S_{i,k \cap j,\ell} \rvert + \lvert S_{i,k \setminus j,\ell} \rvert)(\lvert S_{i,k \cap j,\ell} \rvert + \lvert S_{j,\ell \setminus i,k} \rvert) [a_{n,n-2,k,\ell}\ind(k=j-i) + b_{n,n-2,k,\ell}\ind(k<j-i)] (1-1/n) \\
    &+ [\ind(j \in S_{i,k}) + \ind(j+\ell \in S_{i,k})](\lvert S_{i,k \cap j,\ell} \rvert + \lvert S_{j,\ell \setminus i,k} \rvert) a_{n,n-2,k,\ell}\ind(k<j-i) (1-1/n)\\
    &+ [\ind(i \in S_{j,\ell}) + \ind(i+k \in S_{j,\ell})](\lvert S_{i,k \cap j,\ell} \rvert + \lvert S_{i,k \setminus j,\ell} \rvert) a_{n,n-2,k,\ell}\ind(k<j-i) (1-1/n)\Big)\\
    & -  \Big(\sum_{i=1}^{n-1} \sum_{k=1}^{n-i} \lvert S_{i,k} \rvert c_{n,n-1,k} \Big)^2\Big] + O\Big(\frac{1}{n^2}\Big).
  \end{align*}
\end{lemma}

The following lemma further calculates the expectations of those cardinalities present in Lemma~\ref{lemma:var,cond}.

\begin{lemma}\label{lemma:card}
  Assume $Z_1,\ldots,Z_n,Y_1,Y_2,Y_3,Y_4$ are i.i.d. from a continuous distribution. Define 
  \begin{align*}
    &S_{12} := \sum_{i=1}^n \ind\Big(Y_1 \wedge Y_2 < Z_i < Y_1 \vee Y_2\Big),\\
    &S_{1234} := \sum_{i=1}^n \ind\Big((Y_1 \wedge Y_2) \vee (Y_3 \wedge Y_4) < Z_i < (Y_1 \vee Y_2) \wedge (Y_3 \vee Y_4)\Big),\\
    &S_{12\setminus34} := \sum_{i=1}^n \ind\Big(Y_1 \wedge Y_2 < Z_i < Y_1 \vee Y_2, \Big\{Z_i > Y_3 \vee Y_4 ~{\rm or}~ Z_i < Y_3 \wedge Y_4\Big\}\Big),\\
    &S_{34\setminus12} := \sum_{i=1}^n \ind\Big(Y_3 \wedge Y_4 < Z_i < Y_3 \vee Y_4, \Big\{Z_i > Y_1 \vee Y_2 ~{\rm or}~ Z_i < Y_1 \wedge Y_2\Big\}\Big),\\
    &S_{1223} := \sum_{i=1}^n \ind\Big((Y_1 \wedge Y_2) \vee (Y_2 \wedge Y_3) < Z_i < (Y_1 \vee Y_2) \wedge (Y_2 \vee Y_3)\Big),\\
    &S_{12\setminus23} := \sum_{i=1}^n \ind\Big(Y_1 \wedge Y_2 < Z_i < Y_1 \vee Y_2, \Big\{Z_i > Y_2 \vee Y_3 ~{\rm or}~ Z_i < Y_2 \wedge Y_3\Big\}\Big),\\
    &S_{23\setminus12} := \sum_{i=1}^n \ind\Big(Y_2 \wedge Y_3 < Z_i < Y_2 \vee Y_3, \Big\{Z_i > Y_1 \vee Y_2 ~{\rm or}~ Z_i < Y_1 \wedge Y_2\Big\}\Big).
  \end{align*}
  We then have
  \begin{align*}
    & \E[S_{12}] = \frac{1}{3}n, ~~ \E[S_{12}^2] = \frac{1}{6}n^2 + \frac{1}{6}n,~~ \E[S_{1234}] = \frac{2}{15}n, ~~ \E[S_{1234}^2] = \frac{2}{45}n^2 + \frac{4}{45}n,\\
    & \E[S_{1234}S_{12\setminus34}] = \E[S_{1234}S_{34\setminus12}] = \E[S_{12\setminus34}S_{34\setminus12}] = \frac{1}{45}n^2 - \frac{1}{45}n,~~ \E[S_{1223}] = \frac{1}{6}n,\\
    & \E[S_{1223}^2] = \frac{1}{15}n^2 + \frac{1}{10}n,~~ \E[S_{1223}S_{12\setminus23}] = \E[S_{1223}S_{23\setminus12}] = \E[S_{12\setminus23}S_{23\setminus12}] = \frac{1}{60}n^2 - \frac{1}{60}n,\\
    & \E[\ind(Y_1 \wedge Y_2 < Y_3 < Y_1 \vee Y_2)S_{1234}] = \frac{1}{15}n ,~~ {\rm and}~\E[\ind(Y_1 \wedge Y_2 < Y_3 < Y_1 \vee Y_2)S_{34\setminus12}] = \frac{1}{30}n.
  \end{align*}
\end{lemma}

Applying Lemma~\ref{lemma:card} to Lemma~\ref{lemma:var,cond}, we have the following lemma on the expectation of the conditional variance of $\overline{\xi}_b$.

\begin{lemma}\label{lemma:var} We have
  \begin{align*}
    \limsup_{n \to \infty} n \E[\Var[\overline{\xi}_b \given \mX,\mY]] \le \frac{3}{5} - \frac{8}{5} \frac{1}{e^2} \approx 0.3835 < \frac{2}{5}.
  \end{align*}
\end{lemma}

A more involved analysis of $\Var[\sum_{i=1}^{n-1} \sum_{k=1}^{n-i} \lvert S_{i,k} \rvert c_{n,n-1,k}]$ in Lemma~\ref{lemma:var,cond} can lead to the exact limit value of $n \E[\Var[\overline{\xi}_b \given \mX,\mY]]$ following the same proof techniques as in Lemma~\ref{lemma:var}. However, a limit superior of $n \E[\Var[\overline{\xi}_b \given \mX,\mY]]$ smaller than $2/5$ appears to be sufficient for establishing the bootstrap inconsistency.

Combining Lemmas~\ref{lemma:exp} and \ref{lemma:var}, a proof of Theorem~\ref{thm:main} is finally ready.

\begin{proof}[Proof of Theorem~\ref{thm:main}]

By Lemmas~\ref{lemma:l2}, \ref{lemma:exp}, and \ref{lemma:var}, one has
\begin{align}\label{eq:exp}
  \lim_{n \to \infty} \E[\tilde{\xi}_b] = \frac{1}{e},
\end{align}
and
\begin{align}\label{eq:var}
  \limsup_{n \to \infty} n \E[\Var[\tilde{\xi}_b \given \mX,\mY]] \le \frac{3}{5} - \frac{8}{5} \frac{1}{e^2} \approx 0.3835 < \frac{2}{5}.
\end{align}

Proposition \ref{thm:chatterjee}, on the other hand, shows
\begin{align*}
  \sqrt{n} \xi_n \stackrel{\sf d}{\longrightarrow} N(0,2/5),
\end{align*}
where $\stackrel{\sf d}{\longrightarrow}$ stands for convergence in distribution.

{\bf Proof of the first statement.} For any $\epsilon>0$, since $\sqrt{n} \xi_n$ is bounded in probability by the above central limit theorem, one can find $C_1 = C_1(\epsilon)>0$ such that $\P(\lvert \sqrt{n} \xi_n \rvert > C_1) < \epsilon$ for all $n$ sufficiently large.

Note that for any constant $C_2>0$,
\begin{align*}
  &\E[\tilde{\xi}_b] = \E[\E[\tilde{\xi}_b \given \mX,\mY]] \\
  =& \E[\E[\tilde{\xi}_b \given \mX,\mY] \ind(\E[\tilde{\xi}_b \given \mX,\mY] > C_2)] + \E[\E[\tilde{\xi}_b \given \mX,\mY] \ind(\E[\tilde{\xi}_b \given \mX,\mY] \le C_2)]\\
  \le&  \E[\E[\tilde{\xi}_b \given \mX,\mY] \ind(\E[\tilde{\xi}_b \given \mX,\mY] > C_2)] + C_2.
\end{align*}
By \eqref{eq:exp} and the fact that $\E[\tilde{\xi}_b \given \mX,\mY]$ is universally bounded for all $\mX,\mY$, we can take $C_2<e^{-1}$ and then for all $n$ sufficiently large,
\begin{align*}
  \P(\E[\tilde{\xi}_b \given \mX,\mY] > C_2) \ge 2\epsilon,
\end{align*}
for some $\epsilon>0$.

Then for all $n$ sufficiently large, with probability at least $\epsilon$,
\begin{align*}
  n\E[(\tilde{\xi}_b - \xi_n)^2 \given \mX,\mY] \ge n(C_2^2 - 2C_2C_1/\sqrt{n} ).
\end{align*}
This implies that $n\E[(\tilde{\xi}_b - \xi_n)^2 \given \mX,\mY]$ does not converge in probability to any constant.

If, on the other hand, $n\Var[\tilde{\xi}_b \given \mX,\mY]$ converges to $2/5$ in probability, then by the Portmanteau lemma, 
\begin{align*}
  \liminf_{n \to \infty} n\E[\Var[\tilde{\xi}_b \given \mX,\mY]] \ge \frac{2}{5},
\end{align*}
which contradicts \eqref{eq:var}. Therefore $n\Var[\tilde{\xi}_b \given \mX,\mY]$ does not converge to $2/5$ in probability.

{\bf Proof of the second statement.} We adopt the argument of \cite{abadie2008failure}. For all sufficiently large $n$, we can establish in a similar way as above that, with probability at least $\epsilon$,
\begin{align*}
  \sqrt{n} (\tilde{\xi}_b - \xi_n) \ge \sqrt{n} (C_2 - C_1/\sqrt{n}).
\end{align*}
This shows that $\sqrt{n}(\tilde{\xi}_b - \xi_n)$ cannot converge in distribution to $N(0,2/5)$.

If $\sqrt{n}(\tilde{\xi}_b - \E[\tilde{\xi}_b\given\mX,\mY])$ converges in distribution to $N(0,2/5)$, then 
\[
\liminf_{n\to\infty} n\Var[\tilde{\xi}_b\given\mX,\mY] \ge 2/5 
\]
by the Portmanteau lemma. If the convergence in distribution holds for almost all sequences $X_1,X_2,\ldots$ and $Y_1,Y_2,\ldots$, then 
\[
\liminf_{n\to\infty} n\E[\Var[\tilde{\xi}_b\given\mX,\mY]] \ge 2/5, 
\]
which contradicts \eqref{eq:var}.
\end{proof}

\section{Auxiliary proofs}\label{sec:rest-proof}

\paragraph*{Notation.}

For any two real sequences $\{a_n\}$ and $\{b_n\}$, write $a_n \lesssim b_n$ (or equivalently, $b_n \gtrsim a_n$) if there exists a universal constant $C>0$ such that $|a_n/b_n| \le C$ for all sufficiently large $n$.

\subsection{Proof of Lemma~\ref{lemma:l2}}

\begin{proof}[Proof of Lemma~\ref{lemma:l2}]

Note that for any $i \in \zahl{n-1}$, if $X_{b,[i]_b} = X_{b,[i+1]_b}$, we have $Y_{b,[i]_b} = Y_{b,[i+1]_b}$ and then $R_{b,[i]_b} = R_{b,[i+1]_b}$ from the properties of bootstrap samples. Then $\lvert R_{b,[i+1]_b} - R_{b,[i]_b} \rvert$ is nonzero if and only if $X_{b,[i]_b} < X_{b,[i+1]_b}$. For any $i \in \zahl{n}$ such that $W_{[i]}=0$, $(X_{[i]}, Y_{[i]})$ will not appear in the bootstrap sample. For any $i \in \zahl{n}$ such that $W_{[i]}>0$, the right nearest neighbor of $X_{[i]}$ in the bootstrap sample with strictly larger value is $X_{[i+k]}$, where $k$ is the smallest positive interger such that $W_{[i+k]}>0$. If such $k$ does not exist, $X_{[i]}$ is the largest value in the bootstrap sample. Therefore we can write $\tilde{\xi}_b$ equivalently as
\begin{align*}
  \tilde{\xi}_b &:= 1 - \frac{n}{2\sum_{i=1}^n W_i \tL_i(n-\tL_i)} \sum_{i=1}^{n-1} \ind\Big(W_{[i]}>0\Big)\Big\lvert \tR_{[i+k(i)]} - \tR_{[i]} \Big\rvert,
\end{align*}
where
\begin{align*}
  \tL_i = \sum_{j=1}^n W_j \ind(Y_j \ge Y_i).
\end{align*}

Now since $\tR_i \le n$ for any $i \in \zahl{n}$, we have
\begin{align*}
  \lvert \tilde{\xi}_b - \hat{\xi}_b \rvert =& \Big\lvert \Big(\frac{n}{2\sum_{i=1}^n W_i \tL_i(n-\tL_i)} - \frac{3}{n^2-1} \Big) \sum_{i=1}^{n-1} \ind\Big(W_{[i]}>0\Big)\Big\lvert \tR_{[i+k(i)]} - \tR_{[i]} \Big\rvert \Big\rvert\\
  \le& 2n(n-1) \Big\lvert \frac{n}{2\sum_{i=1}^n W_i \tL_i(n-\tL_i)} - \frac{3}{n^2-1} \Big\rvert.
\end{align*}


By a careful use of the Eforn-Stein inequality, we have the following lemma.

\begin{lemma}\label{lemma:l2var}
  For almost all $\mX = (X_1,\ldots,X_n)$ and $\mY = (Y_1,\ldots,Y_n)$, for any $\epsilon = \epsilon(n)>0$,
  \begin{align*}
    \P\Big(\Big\lvert \sum_{i=1}^n W_i \tL_i(n-\tL_i) - \frac{n^3}{6} \Big\rvert > \epsilon \Biggiven \mX,\mY \Big) \lesssim \epsilon^{-2} n^3.
  \end{align*}
\end{lemma}

Note that if $\lvert \sum_{i=1}^n W_i \tL_i(n-\tL_i) - n^3/6 \rvert \le \epsilon$ for some $\epsilon/n^3 \to 0$, then
\begin{align*}
  \Big\lvert \frac{1}{\sum_{i=1}^n W_i \tL_i(n-\tL_i)} - \frac{6}{n(n^2-1)} \Big\rvert \lesssim \frac{1}{n^6} \Big\lvert n(n^2-1) - 6 \sum_{i=1}^n W_i \tL_i(n-\tL_i) \Big\rvert \lesssim \frac{1}{n^6}(n + \epsilon),
\end{align*}
and then
\begin{align*}
  \lvert \tilde{\xi}_b - \hat{\xi}_b \rvert \lesssim \frac{1}{n^3}\epsilon + \frac{1}{n^2}.
\end{align*}

Also note that we always have $\lvert \tilde{\xi}_b - \hat{\xi}_b \rvert \lesssim 1$. By Lemma~\ref{lemma:l2var}, if $\epsilon/n^3 \to 0$, then for almost all $\mX,\mY$,
\begin{align*}
  &\E [(\tilde{\xi}_b - \hat{\xi}_b)^2 \given \mX, \mY] \lesssim \P\Big(\Big\lvert \sum_{i=1}^n W_i \tL_i(n-\tL_i) - \frac{n^3}{6} \Big\rvert > \epsilon \Biggiven \mX,\mY \Big) + \frac{1}{n^6}\epsilon^2 + \frac{1}{n^4}\\
  \lesssim& n^3 \frac{1}{\epsilon^2} + \frac{1}{n^6}\epsilon^2 + \frac{1}{n^4}.
\end{align*}

Taking $\epsilon \asymp n^{9/4}$, we obtain
\begin{align*}
  \E [(\tilde{\xi}_b - \hat{\xi}_b)^2 \given \mX, \mY] \lesssim n^{-3/2},
\end{align*}
and then the proof is complete.
\end{proof}

\subsection{Proof of Lemma~\ref{lemma:han-l1}}

\begin{proof}[Proof of Lemma~\ref{lemma:han-l1}]

By the definition of $k(i)$ for $i \in \zahl{n}$, we have for any $i \in \zahl{n}$,
\begin{align*}
  \Big\lvert \tR_{[i+k(i)]} - \tR_{[i]} \Big\rvert = \sum_{k=1}^{n-i} \Big\lvert \tR_{[i+k]} - \tR_{[i]} \Big\rvert \ind\Big(W_{[i+\ell]}=0,\forall \ell \in \zahl{k-1}\Big) \ind\Big(W_{[i+k]}>0\Big).
\end{align*}
From the definition of $[\tR_i]_{i=1}^n$, for any $i,k \in 
\zahl{n}$ with $i \neq k$,
\begin{align*}
  \Big\lvert \tR_k - \tR_i \Big\rvert = \Big\lvert \sum_{j=1}^n W_j \ind(Y_j \le Y_k) - \sum_{j=1}^n W_j \ind(Y_j \le Y_i) \Big\rvert = \sum_{j=1}^{n} W_j \ind\big(Y_i \wedge Y_k < Y_j \le Y_i \vee Y_k \big).
\end{align*}
Combining the above two with \eqref{eq:hatxib} yields
\begin{align*}
  \hat{\xi}_b =& 1 - \frac{3}{n^2-1} \sum_{i=1}^{n-1} \ind\Big(W_{[i]}>0\Big)\sum_{k=1}^{n-i}\Big[ \sum_{j=1}^{n} W_j \ind\big(Y_{[i]} \wedge Y_{[i+k]} < Y_j \le Y_{[i]} \vee Y_{[i+k]} \big) \Big] \\
  &\ind\Big(W_{[i+\ell]}=0,\forall \ell \in \zahl{k-1}\Big) \ind\Big(W_{[i+k]}>0\Big).
\end{align*}
We then have
\begin{align*}
 \lvert \overline{\xi}_b - \hat{\xi}_b \rvert \le& \frac{3}{n^2-1} \sum_{i=1}^{n-1} \ind\Big(W_{[i]}>0\Big)\sum_{k=1}^{n-i}\Big( W_{[i]} + W_{[i+k]}\Big)  \ind\Big(W_{[i+\ell]}=0,\forall \ell \in \zahl{k-1}\Big) \ind\Big(W_{[i+k]}>0\Big)\\
 \le& \frac{6}{n^2-1} \sum_{i=1}^n W_i = \frac{6n}{n^2-1} = O\Big(\frac{1}{n}\Big),
\end{align*}
since among all the units with positive weights, each unit only has at most one right NN and can be the right NN of at most one unit. 
\end{proof}

\subsection{Proof of Lemma~\ref{lemma:multi}}

\begin{proof}[Proof of Lemma~\ref{lemma:multi}]
Note that $(W_1,\ldots,W_n) \sim M_n(n;1/n,\ldots,1/n)$, then since $\cS$ and $\cT$ are disjoint,
\begin{align*}
  \sum_{s\in\cS} W_s \Biggiven \sum_{t \in \cT} W_t \sim {\rm Bin}\Big( n - \sum_{t \in \cT} W_t, \frac{\lvert \cS \rvert/n}{1 - \lvert \cT \rvert/n} \Big), ~~ \sum_{t \in \cT} W_t \sim {\rm Bin}\Big(n, \frac{\lvert \cT \rvert}{n}\Big).
\end{align*}
We then obtain
\begin{align*}
  & \E \Big[\Big(\sum_{s\in\cS} W_s\Big) \ind\Big(\sum_{t \in \cT} W_t = 0\Big)\Big] = \E \Big[\sum_{s\in\cS} W_s \Biggiven \sum_{t \in \cT} W_t = 0\Big] \P\Big(\sum_{t \in \cT} W_t = 0 \Big)\\
  =& n \frac{\lvert \cS \rvert/n}{1 - \lvert \cT \rvert/n} \Big(1-\frac{\lvert \cT \rvert}{n}\Big)^n = \lvert \cS \rvert \Big(1-\frac{\lvert \cT \rvert}{n}\Big)^{n-1},
\end{align*}
and
\begin{align*}
  & \E \Big[\Big(\sum_{s\in\cS} W_s\Big)^2 \ind\Big(\sum_{t \in \cT} W_t = 0\Big)\Big] = \E \Big[\Big(\sum_{s\in\cS} W_s\Big)^2 \Biggiven \sum_{t \in \cT} W_t = 0\Big] \P\Big(\sum_{t \in \cT} W_t = 0 \Big)\\
  =& \Big[n \frac{\lvert \cS \rvert/n}{1 - \lvert \cT \rvert/n} + n(n-1) \Big(\frac{\lvert \cS \rvert/n}{1 - \lvert \cT \rvert/n}\Big)^2 \Big]\Big(1-\frac{\lvert \cT \rvert}{n}\Big)^n \\
  =& \lvert \cS \rvert \Big(1-\frac{\lvert \cT \rvert}{n}\Big)^{n-1} + \lvert \cS \rvert^2 \Big(1-\frac{1}{n}\Big) \Big(1-\frac{\lvert \cT \rvert}{n}\Big)^{n-2}.
\end{align*}
Note that from the properties of the multinomial distribution,
\begin{align*}
  \sum_{s\in\cS} W_s, \sum_{s\in\cS'} W_s \Biggiven \sum_{s\in\cS\cup\cS'} W_s \sim M_2\Big(\sum_{s\in\cS\cup\cS'} W_s; \frac{\lvert \cS \rvert}{\lvert \cS \rvert + \lvert \cS' \rvert}, \frac{\lvert \cS' \rvert}{\lvert \cS \rvert + \lvert \cS' \rvert}\Big).
\end{align*}
It then holds that
\begin{align*}
  & \E \Big[\Big(\sum_{s\in\cS} W_s\Big) \Big(\sum_{s\in\cS'} W_s\Big) \ind\Big(\sum_{t \in \cT} W_t = 0\Big)\Big] = \E \Big[\Big(\sum_{s\in\cS} W_s\Big) \Big(\sum_{s\in\cS'} W_s\Big) \Biggiven \sum_{t \in \cT} W_t = 0\Big] \P\Big(\sum_{t \in \cT} W_t = 0 \Big)\\
  =& \E \Big[\E \Big[\Big(\sum_{s\in\cS} W_s\Big) \Big(\sum_{s\in\cS'} W_s\Big) \Biggiven \sum_{s\in\cS\cup\cS'} W_s, \sum_{t \in \cT} W_t = 0\Big] \Biggiven \sum_{t \in \cT} W_t = 0 \Big] \P\Big(\sum_{t \in \cT} W_t = 0 \Big)\\
  =& \E \Big[\Big[\Big(\sum_{s\in\cS\cup\cS'} W_s\Big)^2 - \Big(\sum_{s\in\cS\cup\cS'} W_s\Big)\Big] \frac{\lvert \cS \rvert \lvert \cS' \rvert}{(\lvert \cS \rvert + \lvert \cS' \rvert)^2} \Biggiven \sum_{t \in \cT} W_t = 0 \Big] \P\Big(\sum_{t \in \cT} W_t = 0 \Big)\\
  =& n(n-1) \Big(\frac{(\lvert \cS \rvert + \lvert \cS' \rvert)/n}{1 - \lvert \cT \rvert/n}\Big)^2 \frac{\lvert \cS \rvert \lvert \cS' \rvert}{(\lvert \cS \rvert + \lvert \cS' \rvert)^2} \Big(1-\frac{\lvert \cT \rvert}{n}\Big)^n\\
  =& \lvert \cS \rvert \lvert \cS' \rvert \Big(1-\frac{1}{n}\Big) \Big(1-\frac{\lvert \cT \rvert}{n}\Big)^{n-2}.
\end{align*}

The proof is thus complete.
\end{proof}

\subsection{Proof of Lemma~\ref{lemma:exp,cond}}

\begin{proof}[Proof of Lemma~\ref{lemma:exp,cond}]
For any $i \in \zahl{n-1}$ and $k \in \zahl{n-i}$, consider
\[
  \Big( \sum_{j\in S_{i,k}} W_j  \Big) \ind\Big(W_i>0\Big) \ind\Big(W_{i+\ell}=0,\forall \ell \in \zahl{k-1}\Big) \ind\Big(W_{i+k}>0\Big).
\]
We decompose the above term as
\begin{align*}
  &\ind\Big(W_i>0\Big) \ind\Big(W_{i+\ell}=0,\forall \ell \in \zahl{k-1}\Big) \ind\Big(W_{i+k}>0\Big)\\
  =&\Big[1-\ind\Big(W_i=0\Big)\Big] \ind\Big(W_{i+\ell}=0,\forall \ell \in \zahl{k-1}\Big) \Big[1-\ind\Big(W_{i+k}=0\Big)\Big]\\
  =& \ind\Big(W_{i+\ell}=0,\forall \ell \in \zahl{k-1}\Big) - \ind\Big(W_{i+\ell}=0,\forall \ell \in \zahl{k-1} \cup \{0\}\Big) \\
  &- \ind\Big(W_{i+\ell}=0,\forall \ell \in \zahl{k}\Big) + \ind\Big(W_{i+\ell}=0,\forall \ell \in \zahl{k} \cup \{0\}\Big).
\end{align*}
Note that for any $i \in \zahl{n-1}$ and $k \in \zahl{n-i}$, $S_{i,k}$ is disjoint with $\{i, i+1,\ldots,i+k\}$. Then, by Lemma~\ref{lemma:multi},
\begin{align*}
  &\E\Big[\Big(\sum_{j\in S_{i,k}} W_j\Big) \ind\Big(W_{i+\ell}=0,\forall \ell \in \zahl{k}\Big) \Big] = \lvert S_{i,k} \rvert \Big(1-\frac{k}{n}\Big)^{n-1},\\
  &\E\Big[\Big(\sum_{j\in S_{i,k}} W_j\Big) \ind\Big(W_{i+\ell}=0,\forall \ell \in \zahl{k-1}\Big) \Big] = \lvert S_{i,k} \rvert \Big(1-\frac{k-1}{n}\Big)^{n-1},\\
  &\E\Big[\Big(\sum_{j\in S_{i,k}} W_j\Big) \ind\Big(W_{i+\ell}=0,\forall \ell \in \zahl{k-1} \cup \{0\} \Big) \Big] = \lvert S_{i,k} \rvert \Big(1-\frac{k}{n}\Big)^{n-1},\\
  &\E\Big[\Big(\sum_{j\in S_{i,k}} W_j\Big) \ind\Big(W_{i+\ell}=0,\forall \ell \in \zahl{k} \cup \{0\} \Big) \Big] = \lvert S_{i,k} \rvert \Big(1-\frac{k+1}{n}\Big)^{n-1},
\end{align*}
and accordingly
\begin{align*}
  &\E[\overline{\xi}_b \given \mX,\mY] \\
  =& 1 - \frac{3}{n^2-1} \sum_{i=1}^{n-1} \sum_{k=1}^{n-i} \E \Big[\Big( \sum_{j\in S_{i,k}} W_j  \Big) \ind\Big(W_i>0\Big) \ind\Big(W_{i+\ell}=0,\forall \ell \in \zahl{k-1}\Big) \ind\Big(W_{i+k}>0\Big) \Biggiven \mX,\mY \Big]\\
  =& 1 - \frac{3}{n^2-1} \sum_{i=1}^{n-1} \sum_{k=1}^{n-i} \lvert S_{i,k} \rvert \Big[ \Big(1-\frac{k-1}{n}\Big)^{n-1} - 2 \Big(1-\frac{k}{n}\Big)^{n-1} + \Big(1-\frac{k+1}{n}\Big)^{n-1} \Big].
\end{align*}
This completes the proof.
\end{proof}

\subsection{Proof of Lemma~\ref{lemma:exp}}

\begin{proof}[Proof of Lemma~\ref{lemma:exp}]
We have
\begin{align*}
  & \E[\overline{\xi}_b] = 1 - \frac{3}{n^2-1} \sum_{i=1}^{n-1} \sum_{k=1}^{n-i} \E[\lvert S_{i,k} \rvert] \Big[ \Big(1-\frac{k-1}{n}\Big)^{n-1} - 2 \Big(1-\frac{k}{n}\Big)^{n-1} + \Big(1-\frac{k+1}{n}\Big)^{n-1} \Big]\\
  =& 1 - \frac{1}{n^2-1} \sum_{i=1}^{n-1} \sum_{k=1}^{n-i} (n-k-1) \Big[ \Big(1-\frac{k-1}{n}\Big)^{n-1} - 2 \Big(1-\frac{k}{n}\Big)^{n-1} + \Big(1-\frac{k+1}{n}\Big)^{n-1} \Big].
\end{align*}
Note that
\begin{align*}
  & \sum_{i=1}^{n-1} \sum_{k=1}^{n-i} \Big[ \Big(1-\frac{k-1}{n}\Big)^{n-1} - 2 \Big(1-\frac{k}{n}\Big)^{n-1} + \Big(1-\frac{k+1}{n}\Big)^{n-1} \Big]\\
  =& \sum_{i=1}^{n-1} \Big[ 1 - \Big(1-\frac{1}{n}\Big)^{n-1} - \Big(\frac{i}{n}\Big)^{n-1} + \Big(\frac{i-1}{n}\Big)^{n-1} \Big] \\
  =& (n-1)\Big[ 1 - \Big(1-\frac{1}{n}\Big)^{n-1}\Big] - \Big(1-\frac{1}{n}\Big)^{n-1},
\end{align*}
and
\begin{align*}
  & \sum_{i=1}^{n-1} \sum_{k=1}^{n-i} k \Big[ \Big(1-\frac{k-1}{n}\Big)^{n-1} - 2 \Big(1-\frac{k}{n}\Big)^{n-1} + \Big(1-\frac{k+1}{n}\Big)^{n-1} \Big]\\
  =& \sum_{i=1}^{n-1} \Big[ 1 - (n-i +1) \Big(\frac{i}{n}\Big)^{n-1} + (n-i) \Big(\frac{i-1}{n}\Big)^{n-1} \Big] = n-1-2\frac{\sum_{i=1}^{n-1} i^{n-1}}{n^{n-1}}.
\end{align*}
Note that 
\[
(1-1/n)^n = e^{-1} - (2e)^{-1}n^{-1} + O(n^{-2}) 
\]
and 
\[
\sum_{i=1}^n i^n/n^n \to e/(e-1) ~~~{\rm as}~~~ n \to \infty.
\]
We then obtain
\begin{align*}
  \E[\overline{\xi}_b] =& 1-\frac{1}{n+1} \Big[(n-1)\Big[ 1 - \Big(1-\frac{1}{n}\Big)^{n-1}\Big] - \Big(1-\frac{1}{n}\Big)^{n-1}\Big] + \frac{1}{n^2-1}\Big[n-1-2\frac{\sum_{i=1}^{n-1} i^{n-1}}{n^{n-1}}\Big]\\
  =& 1 - \Big[1-\frac{1}{e} - \Big(3-\frac{1}{2e}\Big)\frac{1}{n}+ O\Big(\frac{1}{n^2}\Big)\Big] = \frac{1}{e} + \Big(3-\frac{1}{2e}\Big)\frac{1}{n} + O\Big(\frac{1}{n^2}\Big).
\end{align*}

The proof is thus complete.
\end{proof}

\subsection{Proof of Lemma~\ref{lemma:var,cond}}

\begin{proof}[Proof of Lemma~\ref{lemma:var,cond}]
We have
\begin{align*}
  &\Var[\overline{\xi}_b \given \mX,\mY] \\
  =& \frac{9}{(n^2-1)^2} \Var\Big[\sum_{i=1}^{n-1} \ind\Big(W_i>0\Big)\sum_{k=1}^{n-i}\Big[ \sum_{j\in S_{i,k}} W_j  \Big]  \ind\Big(W_{i+\ell}=0,\forall \ell \in \zahl{k-1}\Big) \ind\Big(W_{i+k}>0\Big) \Biggiven \mX, \mY\Big].
  \yestag\label{eq:var,cond4}
\end{align*}
In the same way as Lemma~\ref{lemma:exp,cond},
\begin{align*}
  & \Big\{\E\Big[\sum_{i=1}^{n-1} \ind\Big(W_i>0\Big)\sum_{k=1}^{n-i}\Big[ \sum_{j\in S_{i,k}} W_j  \Big]  \ind\Big(W_{i+\ell}=0,\forall \ell \in \zahl{k-1}\Big) \ind\Big(W_{i+k}>0\Big) \Biggiven \mX, \mY\Big]\Big\}^2\\
  =& \Big\{\sum_{i=1}^{n-1} \sum_{k=1}^{n-i} \lvert S_{i,k} \rvert \Big[ \Big(1-\frac{k-1}{n}\Big)^{n-1} - 2 \Big(1-\frac{k}{n}\Big)^{n-1} + \Big(1-\frac{k+1}{n}\Big)^{n-1} \Big] \Big\}^2\\
  =& \Big(\sum_{i=1}^{n-1} \sum_{k=1}^{n-i} \lvert S_{i,k} \rvert c_{n,n-1,k} \Big)^2.
  \yestag\label{eq:var,cond5}
\end{align*}
It then suffices to consider
\begin{align*}
  &\E\Big[\Big\{\sum_{i=1}^{n-1} \ind\Big(W_i>0\Big)\sum_{k=1}^{n-i}\Big[ \sum_{j\in S_{i,k}} W_j  \Big]  \ind\Big(W_{i+\ell}=0,\forall \ell \in \zahl{k-1}\Big) \ind\Big(W_{i+k}>0\Big)\Big\}^2 \Biggiven \mX, \mY\Big]\\
  =& \sum_{i=1}^{n-1} \E\Big[\Big\{\ind\Big(W_i>0\Big)\sum_{k=1}^{n-i}\Big[ \sum_{j\in S_{i,k}} W_j  \Big]  \ind\Big(W_{i+\ell}=0,\forall \ell \in \zahl{k-1}\Big) \ind\Big(W_{i+k}>0\Big)\Big\}^2 \Biggiven \mX, \mY\Big]\\
  &+ 2 \sum_{i,j=1,i<j}^{n-1} \E\Big[\ind\Big(W_i>0\Big)\sum_{k=1}^{n-i}\Big[ \sum_{t\in S_{i,k}} W_t  \Big]  \ind\Big(W_{i+s}=0,\forall s \in \zahl{k-1}\Big) \ind\Big(W_{i+k}>0\Big)\\
  & \ind\Big(W_j>0\Big)\sum_{\ell=1}^{n-j}\Big[ \sum_{t\in S_{j,\ell}} W_t  \Big]  \ind\Big(W_{j+s}=0,\forall s \in \zahl{\ell-1}\Big) \ind\Big(W_{j+\ell}>0\Big) \Biggiven \mX, \mY\Big].
  \yestag\label{eq:var,cond1}
\end{align*}

For the first term in \eqref{eq:var,cond1},
\begin{align*}
  &\E\Big[\Big\{\ind\Big(W_i>0\Big)\sum_{k=1}^{n-i}\Big[ \sum_{j\in S_{i,k}} W_j  \Big]  \ind\Big(W_{i+\ell}=0,\forall \ell \in \zahl{k-1}\Big) \ind\Big(W_{i+k}>0\Big)\Big\}^2 \Biggiven \mX, \mY\Big]\\
  =& \E\Big[\ind\Big(W_i>0\Big)\sum_{k=1}^{n-i}\Big[ \sum_{j\in S_{i,k}} W_j  \Big]^2  \ind\Big(W_{i+\ell}=0,\forall \ell \in \zahl{k-1}\Big) \ind\Big(W_{i+k}>0\Big) \Biggiven \mX, \mY\Big].
\end{align*}
By Lemma~\ref{lemma:multi} and the fact that $S_{i,k}$ is disjoint with $\{i, i+1,\ldots,i+k\}$,
\begin{align*}
  & \E\Big[\Big\{\ind\Big(W_i>0\Big)\sum_{k=1}^{n-i}\Big[ \sum_{j\in S_{i,k}} W_j  \Big]  \ind\Big(W_{i+\ell}=0,\forall \ell \in \zahl{k-1}\Big) \ind\Big(W_{i+k}>0\Big)\Big\}^2 \Biggiven \mX, \mY\Big]\\
  =& \sum_{k=1}^{n-i} \lvert S_{i,k} \rvert \Big[ \Big(1-\frac{k-1}{n}\Big)^{n-1} - 2 \Big(1-\frac{k}{n}\Big)^{n-1} + \Big(1-\frac{k+1}{n}\Big)^{n-1} \Big]\\
  &+ \sum_{k=1}^{n-i} \lvert S_{i,k} \rvert^2 \Big(1-\frac{1}{n}\Big) \Big[ \Big(1-\frac{k-1}{n}\Big)^{n-2} - 2 \Big(1-\frac{k}{n}\Big)^{n-2} + \Big(1-\frac{k+1}{n}\Big)^{n-2} \Big]\\
  =& \sum_{k=1}^{n-i} \Big(\lvert S_{i,k} \rvert c_{n,n-1,k} + \lvert S_{i,k} \rvert^2 c_{n,n-2,k}(1-1/n)\Big),
\end{align*}
and then
\begin{align*}
  & \sum_{i=1}^{n-1} \E\Big[\Big\{\ind\Big(W_i>0\Big)\sum_{k=1}^{n-i}\Big[ \sum_{j\in S_{i,k}} W_j  \Big]  \ind\Big(W_{i+\ell}=0,\forall \ell \in \zahl{k-1}\Big) \ind\Big(W_{i+k}>0\Big)\Big\}^2 \Biggiven \mX, \mY\Big]\\
  =& \sum_{i=1}^{n-1} \sum_{k=1}^{n-i} \Big(\lvert S_{i,k} \rvert c_{n,n-1,k} + \lvert S_{i,k} \rvert^2 c_{n,n-2,k}(1-1/n)\Big).
  \yestag\label{eq:var,cond6}
\end{align*}

For the second term in \eqref{eq:var,cond1}, for any $i,j \in \zahl{n-1}$ with $i<j$,
\begin{align*}
  & \E\Big[\ind\Big(W_i>0\Big)\sum_{k=1}^{n-i}\Big[ \sum_{t\in S_{i,k}} W_t  \Big]  \ind\Big(W_{i+s}=0,\forall s \in \zahl{k-1}\Big) \ind\Big(W_{i+k}>0\Big)\\
  & \ind\Big(W_j>0\Big)\sum_{\ell=1}^{n-j}\Big[ \sum_{t\in S_{j,\ell}} W_t  \Big]  \ind\Big(W_{j+s}=0,\forall s \in \zahl{\ell-1}\Big) \ind\Big(W_{j+\ell}>0\Big) \Biggiven \mX, \mY\Big]\\
  =& \sum_{k=1}^{n-i} \sum_{\ell=1}^{n-j} \E\Big[\ind\Big(W_i>0\Big)\Big[ \sum_{t\in S_{i,k}} W_t  \Big]  \ind\Big(W_{i+s}=0,\forall s \in \zahl{k-1}\Big) \ind\Big(W_{i+k}>0\Big)\\
  & \ind\Big(W_j>0\Big)\Big[ \sum_{t\in S_{j,\ell}} W_t  \Big]  \ind\Big(W_{j+s}=0,\forall s \in \zahl{\ell-1}\Big) \ind\Big(W_{j+\ell}>0\Big) \Biggiven \mX, \mY\Big]\\
  :=&\sum_{k=1}^{n-i} \sum_{\ell=1}^{n-j} T_{i,j,k,\ell}.
\end{align*}

(a) If $k>j-i$, $T_{i,j,k,\ell}$ is zero for any $\ell$ since $\ind(W_{i+s}=0,\forall s \in \zahl{k-1})\ind(W_j>0)=0$.

(b) If $k=j-i$, for any $\ell$,
\begin{align*}
  T_{i,j,k,\ell} =& \E\Big[\Big[ \sum_{t\in S_{i,k}} W_t  \Big]\Big[ \sum_{t\in S_{j,\ell}} W_t  \Big] \ind\Big(W_i>0\Big) \ind\Big(W_{i+s}=0,\forall s \in \zahl{k-1}\Big) \\
  & \ind\Big(W_j>0\Big)  \ind\Big(W_{j+s}=0,\forall s \in \zahl{\ell-1}\Big) \ind\Big(W_{j+\ell}>0\Big) \Biggiven \mX, \mY\Big].
\end{align*}

Note that under the event $\{W_{i+s}=0,\forall s \in \zahl{k-1}\}$ and $\{W_{j+s}=0,\forall s \in \zahl{\ell-1}\}$, we could decompose the above term as
\begin{align*}
  &\sum_{t\in S_{i,k}} W_t = \sum_{t\in S_{i,k \cap j,\ell}} W_t + \sum_{t\in S_{i,k \setminus j,\ell}} W_t + W_j \ind(j \in S_{i,k}) + W_{j+\ell} \ind(j+\ell \in S_{i,k}),\\
  &\sum_{t\in S_{j,\ell}} W_t = \sum_{t\in S_{i,k \cap j,\ell}} W_t + \sum_{t\in S_{j,\ell \setminus i,k}} W_t + W_i \ind(i \in S_{j,\ell}) + W_{i+k} \ind(i+k \in S_{j,\ell}),
\end{align*}
by noticing that $S_{i,k \cap j,\ell}, S_{i,k \setminus j,\ell}, S_{j,\ell \setminus i,k}$ are mutually disjoint, and these three sets are disjoint with $\{i,i+1,\ldots,i+s\}$ and $\{j,j+1,\ldots,j+\ell\}$. The above decomposition also holds for $k \le j-i$. 

If $k=j-i$, then $W_j \ind(j \in S_{i,k}) = W_{i+k} \ind(i+k \in S_{j,\ell}) = 0$.

We also have for the product of indicator functions in $T_{i,j,k,\ell}$,
\begin{align*}
  & \ind\Big(W_i>0\Big) \ind\Big(W_{i+s}=0,\forall s \in \zahl{k-1}\Big) \ind\Big(W_j>0\Big)  \ind\Big(W_{j+s}=0,\forall s \in \zahl{\ell-1}\Big) \ind\Big(W_{j+\ell}>0\Big)\\
  =& \Big[1-\ind\Big(W_i=0\Big)\Big] \Big[1-\ind\Big(W_j=0\Big)\Big] \Big[1-\ind\Big(W_{j+\ell}=0\Big)\Big] \ind\Big(W_{i+s}=0,\forall s \in \zahl{k-1}\Big) \ind\Big(W_{j+s}=0,\forall s \in \zahl{\ell-1}\Big).
\end{align*}

By Lemma~\ref{lemma:multi},
\begin{align*}
  &T_{i,j,k,\ell} \\
  =& \lvert S_{i,k \cap j,\ell} \rvert a_{n,n-1,k,\ell} + \lvert S_{i,k \cap j,\ell} \rvert^2 a_{n,n-2,k,\ell} (1-1/n) \\
  &+ (\lvert S_{i,k \cap j,\ell} \rvert \lvert S_{i,k \setminus j,\ell} \rvert + \lvert S_{i,k \cap j,\ell} \rvert \lvert S_{j,\ell \setminus i,k} \rvert + \lvert S_{i,k \setminus j,\ell} \rvert \lvert S_{j,\ell \setminus i,k} \rvert) a_{n,n-2,k,\ell} (1-1/n) + O(na_{n,n-2,k,\ell})\\
  =& \lvert S_{i,k \cap j,\ell} \rvert a_{n,n-1,k,\ell} + (\lvert S_{i,k \cap j,\ell} \rvert + \lvert S_{i,k \setminus j,\ell} \rvert)(\lvert S_{i,k \cap j,\ell} \rvert + \lvert S_{j,\ell \setminus i,k} \rvert) a_{n,n-2,k,\ell} (1-1/n) + O(na_{n,n-2,k,\ell}).
  \yestag\label{eq:var,cond2}
\end{align*}

(c) If $k<j-i$, for any $\ell$, the only difference with the case where $k=j-i$ is that we now have to decompose the product of indicator functions in $T_{i,j,k,\ell}$ as
\begin{align*}
  & \ind\Big(W_i>0\Big) \ind\Big(W_{i+s}=0,\forall s \in \zahl{k-1}\Big) \ind\Big(W_{i+k}>0\Big) \ind\Big(W_j>0\Big)  \ind\Big(W_{j+s}=0,\forall s \in \zahl{\ell-1}\Big) \ind\Big(W_{j+\ell}>0\Big)\\
  =& \Big[1-\ind\Big(W_i=0\Big)\Big] \Big[1-\ind\Big(W_{i+k}=0\Big)\Big] \Big[1-\ind\Big(W_j=0\Big)\Big] \Big[1-\ind\Big(W_{j+\ell}=0\Big)\Big] \\
  & \ind\Big(W_{i+s}=0,\forall s \in \zahl{k-1}\Big) \ind\Big(W_{j+s}=0,\forall s \in \zahl{\ell-1}\Big).
\end{align*}

By Lemma~\ref{lemma:multi},
\begin{align*}
  &T_{i,j,k,\ell} \\
  =& \lvert S_{i,k \cap j,\ell} \rvert b_{n,n-1,k,\ell} + (\lvert S_{i,k \cap j,\ell} \rvert + \lvert S_{i,k \setminus j,\ell} \rvert)(\lvert S_{i,k \cap j,\ell} \rvert + \lvert S_{j,\ell \setminus i,k} \rvert) b_{n,n-2,k,\ell} (1-1/n)\\
  &+ [\ind(j \in S_{i,k}) + \ind(j+\ell \in S_{i,k})](\lvert S_{i,k \cap j,\ell} \rvert + \lvert S_{j,\ell \setminus i,k} \rvert) a_{n,n-2,k,\ell} (1-1/n)\\
  &+ [\ind(i \in S_{j,\ell}) + \ind(i+k \in S_{j,\ell})](\lvert S_{i,k \cap j,\ell} \rvert + \lvert S_{i,k \setminus j,\ell} \rvert) a_{n,n-2,k,\ell} (1-1/n) + O(a_{n,n-2,k,\ell}).
  \yestag\label{eq:var,cond3}
\end{align*}

Combining \eqref{eq:var,cond2} and \eqref{eq:var,cond3},
\begin{align*}
  & \sum_{i,j=1,i<j}^{n-1} \E\Big[\ind\Big(W_i>0\Big)\sum_{k=1}^{n-i}\Big[ \sum_{t\in S_{i,k}} W_t  \Big]  \ind\Big(W_{i+s}=0,\forall s \in \zahl{k-1}\Big) \ind\Big(W_{i+k}>0\Big)\\
  & \ind\Big(W_j>0\Big)\sum_{\ell=1}^{n-j}\Big[ \sum_{t\in S_{j,\ell}} W_t  \Big]  \ind\Big(W_{j+s}=0,\forall s \in \zahl{\ell-1}\Big) \ind\Big(W_{j+\ell}>0\Big) \Biggiven \mX, \mY\Big]\\
  =& \sum_{i,j=1,i<j}^{n-1} \sum_{k=1}^{n-i} \sum_{\ell=1}^{n-j} \Big(\lvert S_{i,k \cap j,\ell} \rvert [a_{n,n-1,k,\ell}\ind(k=j-i) + b_{n,n-1,k,\ell}\ind(k<j-i)] \\
  & + (\lvert S_{i,k \cap j,\ell} \rvert + \lvert S_{i,k \setminus j,\ell} \rvert)(\lvert S_{i,k \cap j,\ell} \rvert + \lvert S_{j,\ell \setminus i,k} \rvert) [a_{n,n-2,k,\ell}\ind(k=j-i) + b_{n,n-2,k,\ell}\ind(k<j-i)] (1-1/n)\\
  &+ [\ind(j \in S_{i,k}) + \ind(j+\ell \in S_{i,k})](\lvert S_{i,k \cap j,\ell} \rvert + \lvert S_{j,\ell \setminus i,k} \rvert) a_{n,n-2,k,\ell}\ind(k<j-i) (1-1/n)\\
  &+ [\ind(i \in S_{j,\ell}) + \ind(i+k \in S_{j,\ell})](\lvert S_{i,k \cap j,\ell} \rvert + \lvert S_{i,k \setminus j,\ell} \rvert) a_{n,n-2,k,\ell}\ind(k<j-i) (1-1/n)\Big) + O(n^2).
  \yestag\label{eq:var,cond7}
\end{align*}

Plugging \eqref{eq:var,cond5}, \eqref{eq:var,cond6}, \eqref{eq:var,cond7} into \eqref{eq:var,cond4} completes the proof.
\end{proof}

\subsection{Proof of Lemma~\ref{lemma:card}}
\begin{proof}[Proof of Lemma~\ref{lemma:card}]
Note that all the quantities are invariant under probability transformation. By the continuity of the distribution and the probability integral transformation, we can assume without loss of generality that the distribution is the uniform distribution on $[0,1]$.

(a) For $S_{12}$, it is easy to see $\P(Y_1 \wedge Y_2 < Z_i < Y_1 \vee Y_2) = 1/3$ for any $i \in \zahl{n}$, and then $\E[S_{12}] = n/3$. For $i,j \in \zahl{n}$ and $i \neq j$, let $R_i,R_j$ be the rank of $Z_i,Z_j$ in $\{Y_1,Y_2,Z_i,Z_j\}$. Then
\[
  \P(Y_1 \wedge Y_2 < Z_i,Z_j < Y_1 \vee Y_2) = \P(R_i=2,R_j=3)+\P(R_i=3,R_j=2)=2 \times 2/4! = 1/6.
\]
Then
\[
  \E[S_{12}^2] = n(n-1)/6+n/3 = n^2/6+n/6.
\]

(b) For $S_{1234}$, for any $i \in \zahl{n}$, to have $(Y_1 \wedge Y_2) \vee (Y_3 \wedge Y_4) < Z_i < (Y_1 \vee Y_2) \wedge (Y_3 \vee Y_4)$, we have $Y_1,Y_2$ on different side of $Z_i$, as well as $Y_3,Y_4$. Then 
\[
  \P((Y_1 \wedge Y_2) \vee (Y_3 \wedge Y_4) < Z_i < (Y_1 \vee Y_2) \wedge (Y_3 \vee Y_4)\given Z_i) = 4Z_i^2(1-Z_i)^2,
\]
and then 
\[
  \P((Y_1 \wedge Y_2) \vee (Y_3 \wedge Y_4) < Z_i < (Y_1 \vee Y_2) \wedge (Y_3 \vee Y_4)) = 2/15.
\]
Then $\E[S_{1234}]=2n/15$. In the same way, for any $i,j \in \zahl{n}$ and $i \neq j$,
\[
  \P((Y_1 \wedge Y_2) \vee (Y_3 \wedge Y_4) < Z_i,Z_j < (Y_1 \vee Y_2) \wedge (Y_3 \vee Y_4) \given Z_i,Z_j) = 4(Z_i \wedge Z_j)^2(1-Z_i \vee Z_j)^2,
\]
and then 
\[
  \P((Y_1 \wedge Y_2) \vee (Y_3 \wedge Y_4) < Z_i,Z_j < (Y_1 \vee Y_2) \wedge (Y_3 \vee Y_4)) = 2/45.
\]
Then 
\[
  \E[S_{1234}^2] = 2n(n-1)/45+2n/15 = 2n^2/45+4n/45.
\]

(c) For $S_{1234}S_{12\setminus34}$, for any $i,j \in \zahl{n}$ and $i \neq j$, to have $(Y_1 \wedge Y_2) \vee (Y_3 \wedge Y_4) < Z_i < (Y_1 \vee Y_2) \wedge (Y_3 \vee Y_4)$ and $Y_1 \wedge Y_2 < Z_j < Y_1 \vee Y_2, Z_j > Y_3 \vee Y_4 ~{\rm or}~ Z_j < Y_3 \wedge Y_4$, we consider two cases $Z_i<Z_j$ and $Z_i>Z_j$ seperately. If $Z_i<Z_j$, we need one of $Y_1,Y_2$ smaller than $Z_i$, and the other larger than $Z_j$. We also need one of $Y_3,Y_4$ smaller than $Z_i$, and the other between $Z_i$ and $Z_j$. If $Z_i>Z_j$, we need one of $Y_1,Y_2$ smaller than $Z_j$, and the other larger than $Z_i$. We also need one of $Y_3,Y_4$ larger than $Z_i$, and the other between $Z_i$ and $Z_j$. Combining the two cases, we have
\begin{align*}
  & \P((Y_1 \wedge Y_2) \vee (Y_3 \wedge Y_4) < Z_i < (Y_1 \vee Y_2) \wedge (Y_3 \vee Y_4),\\
  & Y_1 \wedge Y_2 < Z_j < Y_1 \vee Y_2, Z_j > Y_3 \vee Y_4 ~{\rm or}~ Z_j < Y_3 \wedge Y_4 \given Z_i,Z_j)\\
  =& 4[\ind(Z_i<Z_j) Z_i^2 (Z_j-Z_i)(1-Z_j) + \ind(Z_j<Z_i) Z_j (Z_i-Z_j) (1-Z_i)^2],
\end{align*}
and then
\begin{align*}
  & \P((Y_1 \wedge Y_2) \vee (Y_3 \wedge Y_4) < Z_i < (Y_1 \vee Y_2) \wedge (Y_3 \vee Y_4),\\
  & Y_1 \wedge Y_2 < Z_j < Y_1 \vee Y_2, Z_j > Y_3 \vee Y_4 ~{\rm or}~ Z_j < Y_3 \wedge Y_4) = 1/45.
\end{align*}
Then
\[
  \E[S_{1234}S_{12\setminus34}] = n(n-1)/45.
\]

(d) From the symmetry, $\E[S_{1234}S_{34\setminus12}]=n(n-1)/45$.

(e) For $S_{12\setminus34}S_{34\setminus12}$, for any $i,j\in\zahl{n}$ and $i\neq j$, to have $Y_1 \wedge Y_2 < Z_i < Y_1 \vee Y_2, Z_i > Y_3 \vee Y_4 ~{\rm or}~ Z_i < Y_3 \wedge Y_4$ and $Y_3 \wedge Y_4 < Z_j < Y_3 \vee Y_4, Z_j > Y_1 \vee Y_2 ~{\rm or}~ Z_j < Y_1 \wedge Y_2$, we consider $Z_i<Z_j$ and $Z_i>Z_j$ seperately. If $Z_i<Z_j$, we need one of $Y_1,Y_2$ smaller than $Z_i$, and the other between $Z_i$ and $Z_j$. We also need one of $Y_3,Y_4$ larger than $Z_j$, and the other between $Z_i$ and $Z_j$. The case when $Z_i>Z_j$ can be analyzed in the same way. Then
\begin{align*}
  & \P(Y_1 \wedge Y_2 < Z_i < Y_1 \vee Y_2, Z_i > Y_3 \vee Y_4 ~{\rm or}~ Z_i < Y_3 \wedge Y_4,\\
  & Y_3 \wedge Y_4 < Z_j < Y_3 \vee Y_4, Z_j > Y_1 \vee Y_2 ~{\rm or}~ Z_j < Y_1 \wedge Y_2 \given Z_i,Z_j)\\
  =& 4[\ind(Z_i<Z_j) Z_i (Z_j-Z_i)^2 (1-Z_j) + \ind(Z_j<Z_i) Z_j (Z_i-Z_j)^2 (1-Z_i)],
\end{align*}
and then
\begin{align*}
  & \P(Y_1 \wedge Y_2 < Z_i < Y_1 \vee Y_2, Z_i > Y_3 \vee Y_4 ~{\rm or}~ Z_i < Y_3 \wedge Y_4,\\
  & Y_3 \wedge Y_4 < Z_j < Y_3 \vee Y_4, Z_j > Y_1 \vee Y_2 ~{\rm or}~ Z_j < Y_1 \wedge Y_2) = 1/45.
\end{align*}

Then
\[
  \E[S_{12\setminus34}S_{34\setminus12}] = n(n-1)/45.
\]

(f) For $S_{1223}$, to have $(Y_1 \wedge Y_2) \vee (Y_2 \wedge Y_3) < Z_i < (Y_1 \vee Y_2) \wedge (Y_2 \vee Y_3)$, we have $Y_1, Y_3$ on the same side of $Z_i$, and different side of $Y_2$. Then
\[
  \P((Y_1 \wedge Y_2) \vee (Y_2 \wedge Y_3) < Z_i < (Y_1 \vee Y_2) \wedge (Y_2 \vee Y_3)\given Z_i) = Z_i(1-Z_i),
\]
and then 
\[
  \P((Y_1 \wedge Y_2) \vee (Y_2 \wedge Y_3) < Z_i < (Y_1 \vee Y_2) \wedge (Y_2 \vee Y_3)) = 1/6.
\]

Then $\E[S_{1223}] = n/6$. We also have
\[
  \P((Y_1 \wedge Y_2) \vee (Y_2 \wedge Y_3) < Z_i,Z_j < (Y_1 \vee Y_2) \wedge (Y_2 \vee Y_3)\given Z_i,Z_j) = (Z_i\wedge Z_j) (1-Z_i \vee Z_j) (Z_i\wedge Z_j + 1-Z_i \vee Z_j),
\]
and then 
\[
  \P((Y_1 \wedge Y_2) \vee (Y_2 \wedge Y_3) < Z_i,Z_j < (Y_1 \vee Y_2) \wedge (Y_2 \vee Y_3)) = 1/15.
\]

Then 
\[
  \E[S_{1223}^2] = n(n-1)/15+n/6 = n^2/15+n/10.
\]

(g) For $S_{1223}S_{12\setminus23}$, by performing similar analysis, 
\begin{align*}
  & \P((Y_1 \wedge Y_2) \vee (Y_2 \wedge Y_3) < Z_i < (Y_1 \vee Y_2) \wedge (Y_2 \vee Y_3),\\
  & Y_1 \wedge Y_2 < Z_j < Y_1 \vee Y_2, Z_j > Y_2 \vee Y_3 ~{\rm or}~ Z_j < Y_2 \wedge Y_3 \given Z_i,Z_j)\\
  =& \ind(Z_i<Z_j) Z_i (Z_j-Z_i)(1-Z_j) + \ind(Z_j<Z_i) Z_j (Z_i-Z_j) (1-Z_i),
\end{align*}
and then
\begin{align*}
  & \P((Y_1 \wedge Y_2) \vee (Y_2 \wedge Y_3) < Z_i < (Y_1 \vee Y_2) \wedge (Y_2 \vee Y_3),\\
  & Y_1 \wedge Y_2 < Z_j < Y_1 \vee Y_2, Z_j > Y_2 \vee Y_3 ~{\rm or}~ Z_j < Y_2 \wedge Y_3) = 1/60.
\end{align*}

Then
\[
  \E[S_{1223}S_{12\setminus23}] = n(n-1)/60.
\]

(h) We can establish in the same way that
\[
  \E[S_{1223}S_{23\setminus12}] = \E[S_{12\setminus23}S_{23\setminus12}] = n(n-1)/60.
\]

(i) For $\ind(Y_1 \wedge Y_2 < Y_3 < Y_1 \vee Y_2)S_{1234}$, to have $Y_1 \wedge Y_2 < Y_3 < Y_1 \vee Y_2$ and $(Y_1 \wedge Y_2) \vee (Y_3 \wedge Y_4) < Z_i < (Y_1 \vee Y_2) \wedge (Y_3 \vee Y_4)$, we condition on $Y_3$ and $Z_i$. If $Y_3<Z_i$, then we need one of $Y_1,Y_2$ smaller than $Y_3$, and the other larger than $Z_i$. We also need $Y_4$ larger than $Z_i$. Then
\begin{align*}
  &\P(Y_1 \wedge Y_2 < Y_3 < Y_1 \vee Y_2, (Y_1 \wedge Y_2) \vee (Y_3 \wedge Y_4) < Z_i < (Y_1 \vee Y_2) \wedge (Y_3 \vee Y_4)\given Z_i,Y_3) \\
  =& 2Y_3(1-Z_i)^2\ind(Y_3 < Z_i) + 2Z_i^2(1-Y_3)\ind(Y_3 > Z_i),
\end{align*}
and then
\[
  \P(Y_1 \wedge Y_2 < Y_3 < Y_1 \vee Y_2, (Y_1 \wedge Y_2) \vee (Y_3 \wedge Y_4) < Z_i < (Y_1 \vee Y_2) \wedge (Y_3 \vee Y_4)) = \frac{1}{15}.
\]

Then
\[
  \E[\ind(Y_1 \wedge Y_2 < Y_3 < Y_1 \vee Y_2)S_{1234}] = n/15.
\]

(j) For $\ind(Y_1 \wedge Y_2 < Y_3 < Y_1 \vee Y_2)S_{34\setminus12}$, to have $Y_1 \wedge Y_2 < Y_3 < Y_1 \vee Y_2$ and $Y_3 \wedge Y_4 < Z_i < Y_3 \vee Y_4, Z_i > Y_1 \vee Y_2 ~{\rm or}~ Z_i < Y_1 \wedge Y_2$. Conditional on $Y_3$ and $Z_i$, if $Y_3<Z_i$, we need one of $Y_1,Y_2$ is smaller than $Y_3$, and the other between $Y_3$ and $Z_1$. We also need $Y_4$ larger than $Z_1$. Then
\begin{align*}
  &\P(Y_1 \wedge Y_2 < Y_3 < Y_1 \vee Y_2, Y_3 \wedge Y_4 < Z_i < Y_3 \vee Y_4, Z_i > Y_1 \vee Y_2 ~{\rm or}~ Z_i < Y_1 \wedge Y_2 \given Z_i,Y_3) \\
  =& 2Y_3(Z_i-Y_3)(1-Z_i)\ind(Y_3 < Z_i) + 2Z_i(Y_3-Z_i)(1-Y_3)\ind(Y_3 > Z_i),
\end{align*}
and then
\[
  \P(Y_1 \wedge Y_2 < Y_3 < Y_1 \vee Y_2, Y_3 \wedge Y_4 < Z_i < Y_3 \vee Y_4, Z_i > Y_1 \vee Y_2 ~{\rm or}~ Z_i < Y_1 \wedge Y_2) = \frac{1}{30}.
\]

Then
\[
  \E[\ind(Y_1 \wedge Y_2 < Y_3 < Y_1 \vee Y_2)S_{34\setminus12}] = n/30.
\]

The proof is thus complete.
\end{proof}

\subsection{Proof of Lemma~\ref{lemma:var}}

\begin{proof}[Proof of Lemma~\ref{lemma:var}]
By Lemma~\ref{lemma:card}, we have 
\[
\E[\lvert S_{i,k} \rvert] = (n-k-1)/3~~~ {\rm and}~~~ \E[\lvert S_{i,k} \rvert^2] = (n-k-1)^2/6 + (n-k-1)/6. 
\]
One could then obtain
\begin{align*}
  \E\Big[\sum_{i=1}^{n-1} \sum_{k=1}^{n-i} \lvert S_{i,k} \rvert c_{n,n-1,k}\Big] =  \sum_{i=1}^{n-1} \sum_{k=1}^{n-i} \Big(\frac{1}{3}(n-k-1)\Big) c_{n,n-1,k} = O(n^2),
  \yestag\label{eq:var1}
\end{align*}
and
\begin{align*}
  \E\Big[\sum_{i=1}^{n-1} \sum_{k=1}^{n-i} \lvert S_{i,k} \rvert^2 c_{n,n-2,k}\Big] = \frac{1}{6} \sum_{i=1}^{n-1} \sum_{k=1}^{n-i} \Big((n-k-1)(n-k)\Big) c_{n,n-2,k} = \frac{1}{6}\Big(1-\frac{1}{e}\Big)n^3 + O(n^2).
  \yestag\label{eq:var2}
\end{align*}

When $k<j-i$, $i,j,k,\ell$ are distinct. Again by Lemma~\ref{lemma:card}, in this case we have
\[
  \E[\lvert S_{i,k \cap j,\ell} \rvert] = \frac{2}{15}(n-k-\ell-2),
\]
and
\[
  \E[(\lvert S_{i,k \cap j,\ell} \rvert + \lvert S_{i,k \setminus j,\ell} \rvert)(\lvert S_{i,k \cap j,\ell} \rvert + \lvert S_{j,\ell \setminus i,k} \rvert)] = \frac{1}{9}(n-k-\ell-2)^2 + \frac{1}{45}(n-k-\ell-2),
\]
and
\begin{align*}
  &\E\{[\ind(j \in S_{i,k}) + \ind(j+\ell \in S_{i,k})](\lvert S_{i,k \cap j,\ell} \rvert + \lvert S_{j,\ell \setminus i,k} \rvert) + [\ind(i \in S_{j,\ell}) + \ind(i+k \in S_{j,\ell})](\lvert S_{i,k \cap j,\ell} \rvert + \lvert S_{i,k \setminus j,\ell} \rvert)\} \\
  =& 4\Big(\frac{1}{15} + \frac{1}{30}\Big) (n-k-\ell-2) = \frac{2}{5} (n-k-\ell-2).
\end{align*}

Then
\begin{align*}
  &\E\Big[\sum_{i,j=1,i<j}^{n-1} \sum_{k=1}^{n-i} \sum_{\ell=1}^{n-j} \lvert S_{i,k \cap j,\ell} \rvert b_{n,n-1,k,\ell}\ind(k<j-i)\Big] = \E\Big[\sum_{i=1}^{n-3}\sum_{j=i+2}^{n-1} \sum_{k=1}^{j-i-1} \sum_{\ell=1}^{n-j} \lvert S_{i,k \cap j,\ell} \rvert b_{n,n-1,k,\ell}\Big]\\
  =& \frac{2}{15} \sum_{i=1}^{n-3}\sum_{j=i+2}^{n-1} \sum_{k=1}^{j-i-1} \sum_{\ell=1}^{n-j} (n-k-\ell-2) b_{n,n-1,k,\ell}\\
  =& \frac{2}{15} \sum_{k=1}^{n-4}\sum_{\ell=1}^{n-k-3} \sum_{j=k+2}^{n-\ell} \sum_{i=1}^{j-k-1} (n-k-\ell-2) b_{n,n-1,k,\ell}\\
  =& \frac{1}{15} \sum_{k=1}^{n-4}\sum_{\ell=1}^{n-k-3} (n-k-\ell) (n-k-\ell-1) (n-k-\ell-2) b_{n,n-1,k,\ell}\\
  =& \frac{1}{15} \Big(1-\frac{2}{e} + \frac{1}{e^2}\Big) n^3 + O(n^2),
  \yestag\label{eq:var3}
\end{align*}
and
\begin{align*}
  & \E\Big[\sum_{i,j=1,i<j}^{n-1} \sum_{k=1}^{n-i} \sum_{\ell=1}^{n-j} (\lvert S_{i,k \cap j,\ell} \rvert + \lvert S_{i,k \setminus j,\ell} \rvert)(\lvert S_{i,k \cap j,\ell} \rvert + \lvert S_{j,\ell \setminus i,k} \rvert) b_{n,n-2,k,\ell}\ind(k<j-i)\Big]\\
  =& \E\Big[\sum_{i=1}^{n-3}\sum_{j=i+2}^{n-1} \sum_{k=1}^{j-i-1} \sum_{\ell=1}^{n-j} (\lvert S_{i,k \cap j,\ell} \rvert + \lvert S_{i,k \setminus j,\ell} \rvert)(\lvert S_{i,k \cap j,\ell} \rvert + \lvert S_{j,\ell \setminus i,k} \rvert) b_{n,n-2,k,\ell}\Big]\\
  =& \frac{1}{9} \sum_{i=1}^{n-3}\sum_{j=i+2}^{n-1} \sum_{k=1}^{j-i-1} \sum_{\ell=1}^{n-j} (n-k-\ell-2)^2 b_{n,n-2,k,\ell} + \frac{1}{45} \sum_{i=1}^{n-3}\sum_{j=i+2}^{n-1} \sum_{k=1}^{j-i-1} \sum_{\ell=1}^{n-j} (n-k-\ell-2) b_{n,n-2,k,\ell}\\
  =& \frac{1}{18} \sum_{k=1}^{n-4}\sum_{\ell=1}^{n-k-3} (n-k-\ell) (n-k-\ell-1) (n-k-\ell-2)^2 b_{n,n-2,k,\ell} \\
  & + \frac{1}{90} \sum_{k=1}^{n-4}\sum_{\ell=1}^{n-k-3} (n-k-\ell) (n-k-\ell-1) (n-k-\ell-2) b_{n,n-2,k,\ell}\\
  =& \frac{1}{18} \Big[\Big(1-\frac{2}{e} + \frac{1}{e^2}\Big) n^4 + \Big(-13 + \frac{15}{e}-\frac{3}{e^2}\Big)n^3 \Big] + \frac{1}{90} \Big(1-\frac{2}{e} + \frac{1}{e^2}\Big) n^3 + O(n^2),
  \yestag\label{eq:var4}
\end{align*}
and
\begin{align*}
  & \E\Big[\sum_{i,j=1,i<j}^{n-1} \sum_{k=1}^{n-i} \sum_{\ell=1}^{n-j}\Big( [\ind(j \in S_{i,k}) + \ind(j+\ell \in S_{i,k})](\lvert S_{i,k \cap j,\ell} \rvert + \lvert S_{j,\ell \setminus i,k} \rvert) a_{n,n-2,k,\ell}\ind(k<j-i) \\
  &+ [\ind(i \in S_{j,\ell}) + \ind(i+k \in S_{j,\ell})](\lvert S_{i,k \cap j,\ell} \rvert + \lvert S_{i,k \setminus j,\ell} \rvert) a_{n,n-2,k,\ell}\ind(k<j-i) \Big)\Big]\\
  =& \frac{2}{5} \sum_{i=1}^{n-3}\sum_{j=i+2}^{n-1} \sum_{k=1}^{j-i-1} \sum_{\ell=1}^{n-j} (n-k-\ell-2) a_{n,n-2,k,\ell}\\
  =& \frac{1}{5} \sum_{k=1}^{n-4}\sum_{\ell=1}^{n-k-3} (n-k-\ell) (n-k-\ell-1) (n-k-\ell-2) a_{n,n-2,k,\ell}\\
  =& \frac{1}{5} \Big(1-\frac{1}{e}\Big) n^3 + O(n^2).
  \yestag\label{eq:var5}
\end{align*}

When $k=j-i$, we have $j = i+k$ and then by Lemma~\ref{lemma:card},
\[
  \E[\lvert S_{i,k \cap j,\ell} \rvert] = \frac{1}{6}(n-k-\ell-1),
\]
and
\[
  \E[(\lvert S_{i,k \cap j,\ell} \rvert + \lvert S_{i,k \setminus j,\ell} \rvert)(\lvert S_{i,k \cap j,\ell} \rvert + \lvert S_{j,\ell \setminus i,k} \rvert)] = \frac{7}{60}(n-k-\ell-1)^2 + \frac{1}{20}(n-k-\ell-1).
\]

Then
\begin{align*}
  &\E\Big[\sum_{i,j=1,i<j}^{n-1} \sum_{k=1}^{n-i} \sum_{\ell=1}^{n-j} \lvert S_{i,k \cap j,\ell} \rvert a_{n,n-1,k,\ell}\ind(k=j-i)\Big] = \E\Big[\sum_{i=1}^{n-1} \sum_{k=1}^{n-i} \sum_{\ell=1}^{n-i-k} \lvert S_{i,k \cap i+k,\ell} \rvert a_{n,n-1,k,\ell}\Big]\\
  =& \frac{1}{6} \sum_{i=1}^{n-1} \sum_{k=1}^{n-i} \sum_{\ell=1}^{n-i-k} (n-k-\ell-1) a_{n,n-1,k,\ell} = \frac{1}{6} \sum_{k=1}^{n-2} \sum_{\ell=1}^{n-k-1} \sum_{i=1}^{n-k-\ell} (n-k-\ell-1) a_{n,n-1,k,\ell}\\
  =& \frac{1}{6} \sum_{k=1}^{n-2} \sum_{\ell=1}^{n-k-1} (n-k-\ell) (n-k-\ell-1) a_{n,n-1,k,\ell} = O(n^2),
  \yestag\label{eq:var6}
\end{align*}
and
\begin{align*}
  & \E\Big[\sum_{i,j=1,i<j}^{n-1} \sum_{k=1}^{n-i} \sum_{\ell=1}^{n-j} (\lvert S_{i,k \cap j,\ell} \rvert + \lvert S_{i,k \setminus j,\ell} \rvert)(\lvert S_{i,k \cap j,\ell} \rvert + \lvert S_{j,\ell \setminus i,k} \rvert) a_{n,n-2,k,\ell}\ind(k=j-i)\Big]\\
  =& \E\Big[\sum_{i=1}^{n-1} \sum_{k=1}^{n-i} \sum_{\ell=1}^{n-i-k} (\lvert S_{i,k \cap i+k,\ell} \rvert + \lvert S_{i,k \setminus i+k,\ell} \rvert)(\lvert S_{i,k \cap i+k,\ell} \rvert + \lvert S_{i+k,\ell \setminus i,k} \rvert) a_{n,n-2,k,\ell}\Big]\\
  =& \frac{7}{60} \sum_{i=1}^{n-1} \sum_{k=1}^{n-i} \sum_{\ell=1}^{n-i-k} (n-k-\ell-1)^2 a_{n,n-2,k,\ell} + \frac{1}{20} \sum_{i=1}^{n-1} \sum_{k=1}^{n-i} \sum_{\ell=1}^{n-i-k}(n-k-\ell-1) a_{n,n-2,k,\ell}\\
  =& \frac{7}{60} \sum_{k=1}^{n-2} \sum_{\ell=1}^{n-k-1} (n-k-\ell) (n-k-\ell-1)^2 a_{n,n-2,k,\ell} \\
  & + \frac{1}{20} \sum_{k=1}^{n-2} \sum_{\ell=1}^{n-k-1} (n-k-\ell) (n-k-\ell-1) a_{n,n-2,k,\ell}\\
  =& \frac{7}{60} \Big(1-\frac{1}{e}\Big) n^3 + O(n^2).
  \yestag\label{eq:var7}
\end{align*}

Combining all the pieces from \eqref{eq:var1} to \eqref{eq:var7}, we have
\begin{align*}
  & \sum_{i=1}^{n-1} \sum_{k=1}^{n-i} \Big(\lvert S_{i,k} \rvert c_{n,n-1,k} + \lvert S_{i,k} \rvert^2 c_{n,n-2,k}(1-1/n)\Big) \\
  & + 2 \sum_{i,j=1,i<j}^{n-1} \sum_{k=1}^{n-i} \sum_{\ell=1}^{n-j} \Big(\lvert S_{i,k \cap j,\ell} \rvert [a_{n,n-1,k,\ell}\ind(k=j-i) + b_{n,n-1,k,\ell}\ind(k<j-i)] \\
  & + (\lvert S_{i,k \cap j,\ell} \rvert + \lvert S_{i,k \setminus j,\ell} \rvert)(\lvert S_{i,k \cap j,\ell} \rvert + \lvert S_{j,\ell \setminus i,k} \rvert) [a_{n,n-2,k,\ell}\ind(k=j-i) + b_{n,n-2,k,\ell}\ind(k<j-i)] (1-1/n) \\
  &+ [\ind(j \in S_{i,k}) + \ind(j+\ell \in S_{i,k})](\lvert S_{i,k \cap j,\ell} \rvert + \lvert S_{j,\ell \setminus i,k} \rvert) a_{n,n-2,k,\ell}\ind(k<j-i) (1-1/n)\\
  &+ [\ind(i \in S_{j,\ell}) + \ind(i+k \in S_{j,\ell})](\lvert S_{i,k \cap j,\ell} \rvert + \lvert S_{i,k \setminus j,\ell} \rvert) a_{n,n-2,k,\ell}\ind(k<j-i) (1-1/n)\Big)\\
  =& \frac{1}{6}\Big(1-\frac{1}{e}\Big)n^3 + \frac{2}{15} \Big(1-\frac{2}{e} + \frac{1}{e^2}\Big) n^3 + \frac{7}{30} \Big(1-\frac{1}{e}\Big) n^3 + \frac{1}{9} \Big(1-\frac{2}{e} + \frac{1}{e^2}\Big) n^4 - \frac{1}{9} \Big(1-\frac{2}{e} + \frac{1}{e^2}\Big) n^3 \\
  & + \frac{1}{9} \Big(-13 + \frac{15}{e}-\frac{3}{e^2}\Big)n^3 + \frac{1}{45} \Big(1-\frac{2}{e} + \frac{1}{e^2}\Big) n^3 + \frac{2}{5} \Big(1-\frac{1}{e}\Big) n^3 + O(n^2)\\
  =& \frac{1}{9} \Big(1-\frac{2}{e} + \frac{1}{e^2}\Big) n^4 + \Big(-\frac{3}{5} + \frac{7}{9}\frac{1}{e} - \frac{13}{45} \frac{1}{e^2}\Big) n^3 + O(n^2).
  \yestag\label{eq:var8}
\end{align*}

By the proof of Lemma~\ref{lemma:exp}, 
\begin{align*}
  & \E\Big[\Big(\sum_{i=1}^{n-1} \sum_{k=1}^{n-i} \lvert S_{i,k} \rvert c_{n,n-1,k} \Big)^2\Big] \ge \Big[\E\Big(\sum_{i=1}^{n-1} \sum_{k=1}^{n-i} \lvert S_{i,k} \rvert c_{n,n-1,k} \Big)\Big]^2 \\
  =& \frac{1}{9}\Big[\Big(1-\frac{1}{e}\Big)n^2 - \Big(3-\frac{1}{2e}\Big)n+ O(1) \Big]^2 = \frac{1}{9} \Big(1-\frac{1}{e}\Big)^2 n^4 - \frac{2}{9}\Big(1-\frac{1}{e}\Big)\Big(3-\frac{1}{2e}\Big)n^3 + O(n^2).
  \yestag\label{eq:var9}
\end{align*}

Plugging \eqref{eq:var8} and \eqref{eq:var9} to Lemma~\ref{lemma:var,cond} then yields
\begin{align*}
  & \E[\Var[\overline{\xi}_b \given \mX,\mY]] \le \frac{9}{n^4}\Big(1+O(n^{-2})\Big) \\
  & \Big[\frac{1}{9} \Big(1-\frac{2}{e} + \frac{1}{e^2}\Big) n^4 + \Big(-\frac{3}{5} + \frac{7}{9}\frac{1}{e} - \frac{13}{45} \frac{1}{e^2}\Big) n^3 - \frac{1}{9} \Big(1-\frac{1}{e}\Big)^2 n^4 + \frac{2}{9}\Big(1-\frac{1}{e}\Big)\Big(3-\frac{1}{2e}\Big)n^3 + O(n^2) \Big] + O\Big(\frac{1}{n^2}\Big)\\
  =& \Big(\frac{3}{5} - \frac{8}{5} \frac{1}{e^2}\Big) \frac{1}{n} + O\Big(\frac{1}{n^2}\Big),
\end{align*}
and the proof is thus complete.
\end{proof}

\subsection{Proof of Lemma~\ref{lemma:l2var}}

\begin{proof}[Proof of Lemma~\ref{lemma:l2var}]
From the definition of $[\tL_i]_{i=1}^n$, we have
\begin{align*}
  \sum_{i=1}^n W_i \tL_i(n-\tL_i) =& \sum_{i=1}^n W_i \Big(\sum_{j=1}^n W_j \ind(Y_j \ge Y_i) \Big) \Big(n - \sum_{j=1}^n W_j \ind(Y_j \ge Y_i)\Big) \\
  =& \sum_{i=1}^n W_i \Big(\sum_{j=1}^n W_j \ind(Y_j \ge Y_i) \Big) \Big(\sum_{j=1}^n W_j \ind(Y_j < Y_i)\Big).
\end{align*}

We first study the mean. Note that $(W_1,\ldots,W_n)$ is independent of $\mX$ and $\mY$. We can assume $[Y_i]_{i=1}^n$ is strictly increasing with probability one without loss of generality conditional on $\mX$ and $\mY$. Then 
\begin{align*}
  &\E\Big[\sum_{i=1}^n W_i \Big(\sum_{j=1}^n W_j \ind(Y_j \ge Y_i) \Big) \Big(\sum_{j=1}^n W_j \ind(Y_j < Y_i)\Big) \Biggiven \mX, \mY \Big] \\
  =& \E\Big[ \sum_{i=1}^n W_i \Big( n-\sum_{j=1}^{i-1} W_j \Big) \Big(\sum_{j=1}^{i-1} W_j\Big)\Big] \\
  =& n\sum_{i=1}^n \E\Big[W_i\Big(\sum_{j=1}^{i-1} W_j\Big)\Big] - \sum_{i=1}^n \E\Big[W_i\Big(\sum_{j=1}^{i-1} W_j\Big)^2\Big].
\end{align*}

Since $(W_1,\ldots,W_n) \sim M_n(n;1/n,\ldots,1/n)$, for any $i \in 
\zahl{n}$, we have 
\[
\sum_{j=1}^{i-1}W_j \given W_i \sim {\rm Bin}(n-W_i,\frac{i-1}{n-1}) ~~~{\rm and}~~~ W_i \sim {\rm Bin}(n,\frac{1}{n}). 
\]
Then for any $i \in \zahl{n}$,
\begin{align*}
  \E\Big[W_i\Big(\sum_{j=1}^{i-1} W_j\Big)\Big] = \E\Big[\frac{i-1}{n-1}W_i(n-W_i)\Big] = (i-1)(1-1/n),
\end{align*}
and
\begin{align*}
  & \E\Big[W_i\Big(\sum_{j=1}^{i-1} W_j\Big)^2\Big] = \E\Big[\frac{(i-1)(n-i)}{(n-1)^2}W_i(n-W_i)\Big] + \E\Big[\frac{(i-1)^2}{(n-1)^2}W_i(n-W_i)^2\Big]\\
  =& \frac{(i-1)(n-i)}{n} + \frac{(i-1)^2 (n^2-2n+2)}{n^2} = (i-1)(1-1/n) + (i-1)^2 (1-1/n)(1-2/n).
\end{align*}

It then holds that 
\begin{align*}
  &\E\Big[\sum_{i=1}^n W_i \Big(\sum_{j=1}^n W_j \ind(Y_j \ge Y_i) \Big) \Big(\sum_{j=1}^n W_j \ind(Y_j < Y_i)\Big) \Biggiven \mX, \mY \Big]\\
  =& \frac{(n-1)^2 (n^2+2n-2)}{6n} = \frac{1}{6}(n^3 - 5n+6-2n^{-1}).
  \yestag\label{eq:l2var1}
\end{align*}

We next consider the variance. Note that
\begin{align*}
  &\Var\Big[\sum_{i=1}^n W_i \Big(\sum_{j=1}^n W_j \ind(Y_j \ge Y_i) \Big) \Big(\sum_{j=1}^n W_j \ind(Y_j < Y_i)\Big) \Biggiven \mX, \mY \Big] \\
  =& \Var\Big[ \sum_{i=1}^n W_i \Big( n-\sum_{j=1}^{i-1} W_j \Big) \Big(\sum_{j=1}^{i-1} W_j\Big)\Big] \\
  =& \Var\Big[\sum_{i=1}^n W_i \Big( \sum_{j=i}^{n} W_j \Big) \Big(\sum_{j=1}^{i-1} W_j\Big)\Big].
\end{align*}

We can consider $Z_1,Z_2,\ldots,Z_n$ i.i.d. with distribution $\P(Z_1 = k) = 1/n$ for any $k \in \zahl{n}$. Then $(W_1,\ldots,W_n)$ has the same law with $(\sum_{k=1}^n \ind(Z_k=1), \ldots, \sum_{k=1}^n \ind(Z_k=n))$.

To apply the Efron-Stein inequality, consider $\tZ_1$ with the same distribution as $Z_1$ and is independent of $Z_1,\ldots,Z_n$. Let $(\tW_1, \ldots, \tW_n)$ be the new values by replacing $Z_1$ with $\tZ_1$. Assume $Z_1 = k$ and $\tZ_1 = \ell$ for some $k,\ell \in \zahl{n}$. If $k = \ell$, then $W_i = \tW_i$ for all $i \in \zahl{n}$, and then the difference is zero. Then we assume $k<\ell$ without loss of generality. In this case, $\tW_k = W_k-1$, $\tW_\ell = W_\ell + 1$, and $\tW_i = W_i$ for other $i \in \zahl{n}$. Thus we have the following results for different $i \in \zahl{n}$,
\begin{itemize}
  \item $i = k$, 
  \begin{align*}
    &\tW_i \Big( \sum_{j=i}^{n} \tW_j \Big) \Big(\sum_{j=1}^{i-1} \tW_j\Big) - W_i \Big( \sum_{j=i}^{n} W_j \Big) \Big(\sum_{j=1}^{i-1} W_j\Big) \\
    =& (W_k-1) \Big( \sum_{j=k}^{n} W_j \Big) \Big(\sum_{j=1}^{k-1} W_j\Big) - W_k \Big( \sum_{j=k}^{n} W_j \Big) \Big(\sum_{j=1}^{k-1} W_j\Big) \\
    =& -\Big( \sum_{j=k}^{n} W_j \Big) \Big(\sum_{j=1}^{k-1} W_j\Big).
  \end{align*}
  \item $i = \ell$,
  \begin{align*}
    &\tW_i \Big( \sum_{j=i}^{n} \tW_j \Big) \Big(\sum_{j=1}^{i-1} \tW_j\Big) - W_i \Big( \sum_{j=i}^{n} W_j \Big) \Big(\sum_{j=1}^{i-1} W_j\Big) \\
    =& (W_\ell+1) \Big( \sum_{j=\ell}^{n} W_j + 1\Big) \Big(\sum_{j=1}^{\ell-1} W_j - 1\Big) - W_\ell \Big( \sum_{j=\ell}^{n} W_j \Big) \Big(\sum_{j=1}^{\ell-1} W_j\Big)\\
    =& \Big( \sum_{j=\ell}^{n} W_j + 1\Big) \Big(\sum_{j=1}^{\ell-1} W_j - 1\Big) + W_\ell \Big(\sum_{j=1}^{\ell-1} W_j - \sum_{j=\ell}^{n} W_j - 1\Big).
  \end{align*}
  \item $k<i<\ell$,
  \begin{align*}
    &\tW_i \Big( \sum_{j=i}^{n} \tW_j \Big) \Big(\sum_{j=1}^{i-1} \tW_j\Big) - W_i \Big( \sum_{j=i}^{n} W_j \Big) \Big(\sum_{j=1}^{i-1} W_j\Big) \\
    =& W_i \Big( \sum_{j=i}^{n} W_j +1 \Big) \Big(\sum_{j=1}^{i-1} W_j - 1\Big) - W_i \Big( \sum_{j=i}^{n} W_j \Big) \Big(\sum_{j=1}^{i-1} W_j\Big)\\
    =& W_i \Big(\sum_{j=1}^{i-1} W_j - \sum_{j=i}^{n} W_j - 1\Big).
  \end{align*}
\end{itemize}

For other $i \in \zahl{n}$, the difference is zero. 

We thus have
\begin{align*}
  &\sum_{i=1}^n \tW_i \Big( \sum_{j=i}^{n} \tW_j \Big) \Big(\sum_{j=1}^{i-1} \tW_j\Big) - \sum_{i=1}^n W_i \Big( \sum_{j=i}^{n} W_j \Big) \Big(\sum_{j=1}^{i-1} W_j\Big)\\
  =& -\Big( \sum_{j=k}^{n} W_j \Big) \Big(\sum_{j=1}^{k-1} W_j\Big) + \Big( \sum_{j=\ell}^{n} W_j + 1\Big) \Big(\sum_{j=1}^{\ell-1} W_j - 1\Big) + \sum_{i=k+1}^\ell W_i \Big(\sum_{j=1}^{i-1} W_j - \sum_{j=i}^{n} W_j - 1\Big)\\
  =& \Big[-\Big( \sum_{j=k}^{\ell-1} W_j \Big) \Big(\sum_{j=1}^{k-1} W_j\Big) -\Big( \sum_{j=\ell}^{n} W_j \Big) \Big(\sum_{j=1}^{k-1} W_j\Big)\Big] + \Big[ \Big( \sum_{j=\ell}^{n} W_j \Big) \Big(\sum_{j=1}^{k-1} W_j\Big) \\
  & + \Big( \sum_{j=\ell}^{n} W_j \Big) \Big(\sum_{j=k}^{\ell-1} W_j\Big) - \sum_{j=\ell}^{n} W_j + \sum_{j=1}^{k-1} W_j + \sum_{j=k}^{\ell-1} W_j - 1\Big] + \Big[ \sum_{i=k+1}^{\ell} W_i  \Big(\sum_{j=1}^{k-1} W_j\Big)\\
  &+ \sum_{i=k+1}^{\ell} W_i  \Big(\sum_{j=k}^{i-1} W_j\Big) - \sum_{i=k+1}^{\ell} W_i \Big(\sum_{j=i}^{\ell-1} W_j\Big) - \sum_{i=k+1}^{\ell} W_i  \Big(\sum_{j=\ell}^{n} W_j\Big) - \sum_{i=k+1}^{\ell} W_i \Big].
\end{align*}

Note that
\begin{align*}
  \sum_{i=k+1}^\ell W_i \Big(\sum_{j=k}^{i-1} W_j \Big) = \sum_{i=k}^{\ell-1} W_i \Big(\sum_{j=i+1}^{\ell} W_j \Big),~~ \sum_{i=1}^n W_i = n.
\end{align*}

By some algebra,
\begin{align*}
  &\sum_{i=1}^n \tW_i \Big( \sum_{j=i}^{n} \tW_j \Big) \Big(\sum_{j=1}^{i-1} \tW_j\Big) - \sum_{i=1}^n W_i \Big( \sum_{j=i}^{n} W_j \Big) \Big(\sum_{j=1}^{i-1} W_j\Big)\\
  =& (W_\ell - W_k + 1) \Big( \sum_{j=1}^{k-1}W_j - \sum_{j=\ell}^n W_j -1 \Big) + (W_k + W_\ell)\Big( \sum_{j=k}^{\ell-1} W_j\Big) - \sum_{j=k}^{\ell-1} W_j^2.
\end{align*}

Using the fact that $\sum_{i=1}^n W_i = n$ and all $W_i$'s are nonnegative, we have
\begin{align*}
  &\Big\lvert \sum_{i=1}^n \tW_i \Big( \sum_{j=i}^{n} \tW_j \Big) \Big(\sum_{j=1}^{i-1} \tW_j\Big) - \sum_{i=1}^n W_i \Big( \sum_{j=i}^{n} W_j \Big) \Big(\sum_{j=1}^{i-1} W_j\Big) \Big\rvert\\
  \le& (W_\ell + W_k + 1)(n+1) + (W_\ell + W_k) n + \sum_{i=1}^n W_i^2
\end{align*}

Note that $Z_1,\ldots, Z_n, \tZ_1$ are i.i.d, and $W_k \sim {\rm Bin}(n,\frac{1}{n})$ for any $k \in \zahl{n}$. Then from the Efron-Stein inequality,
\begin{align*}
  \Var\Big[\sum_{i=1}^n W_i \Big(\sum_{j=1}^n W_j \ind(Y_j \ge Y_i) \Big) \Big(\sum_{j=1}^n W_j \ind(Y_j < Y_i)\Big) \Biggiven \mX, \mY \Big] \lesssim n^3.
  \yestag\label{eq:l2var2}
\end{align*}

From the Chebyshev inequality, and the bias–variance decomposition using \eqref{eq:l2var1} and \eqref{eq:l2var2}, for any $\epsilon>0$, for almost all $\mX,\mY$,
\begin{align*}
  &\P\Big(\Big\lvert \sum_{i=1}^n W_i \tL_i(n-\tL_i) - \frac{n^3}{6} \Big\rvert \ge \epsilon \Biggiven \mX,\mY \Big)\\
  \le& \epsilon^{-2} \Big[ \Var\Big[\sum_{i=1}^n W_i \tL_i(n-\tL_i) \Biggiven \mX, \mY\Big] + \E\Big[ \Big(\sum_{i=1}^n W_i \tL_i(n-\tL_i) - \frac{n^3}{6}\Big)^2 \Biggiven \mX,\mY \Big] \Big]\\
  \lesssim& \epsilon^{-2} (n^3 + n^2) \lesssim \epsilon^{-2} n^3,
\end{align*}
and then the proof is complete.
\end{proof}

\section*{Acknowledgement}

The authors would like to thank Sourav Chatterjee for confirming the failure of the bootstrap for his rank correlation and for generously sharing his insights with the authors. The authors would also like to thank Peter Bickel for pointing out the works of Beran and for discussing the relation between bootstrap inconsistency and adaptive estimators, to thank Mathias Drton for sharing his work \citep{drton2011quantifying} that led the authors to \cite{beran1997diagnosing} and \cite{samworth2003note}, and to thank Andres Santos for explaining his findings in \cite{fang2019inference}. The authors also benefited from discussions with Lihua Lei, Bodhisattva Sen, and Jon Wellner. 

 This work was partly motivated by discussions with Mona Azadkia, David Childers, Peng Ding, Andreas Hagemann, and Mauricio Olivares. In particular, Peng posed a question to the second author regarding the practical relevance of the results presented in \cite{lin2022limit} and noted that bootstrap is commonly used in practice. This paper serves as a response to Peng's inquiry.

{
\bibliographystyle{apalike}
\bibliography{AMS}
}

\end{document}